\documentclass[a4paper, 11pt, english]{article}
\usepackage{lscape}
\usepackage{pdflscape}

\usepackage{stmaryrd}

\usepackage{cmll}

\usepackage{comment}

\usepackage{amscd}
\usepackage{amssymb,amsmath,amsthm}
\usepackage{mathptmx}
\usepackage{mathrsfs}
\usepackage{color}
\usepackage{xspace}
\usepackage{bussproofs}
\EnableBpAbbreviations

\usepackage{tikz}

  \usepackage{txfonts}

\usepackage{enumerate}

\usepackage{centernot}

\usepackage{mathtools}

\usepackage{hyperref}

\usepackage{cleveref}

\usepackage{relsize}

\usepackage{caption}

\usepackage{multirow}

\usepackage{multicol}

\usepackage{array}
\newcolumntype{C}[1]{>{\centering\arraybackslash}p{#1}}
\newcolumntype{L}[1]{>{\arraybackslash}p{#1}}

\usepackage[left=3cm,top=3cm,right=3cm,bottom=2cm]{geometry}
\def\fCenter{{\mbox{$\ \vdash\ $}}}

\newcommand{\fns}{\footnotesize}

\def\mc{\multicolumn}

\newcommand{\msf}{\mathsf}

\newcommand{\xtop}{\ensuremath{\top}\xspace}
\newcommand{\xbot}{\ensuremath{\bot}\xspace}
\newcommand{\xand}{\ensuremath{\wedge}\xspace}
\newcommand{\xor}{\ensuremath{\vee}\xspace}
\newcommand{\xrarr}{\ensuremath{\rightarrow}\xspace}
\newcommand{\xcrarr}{\ensuremath{\,{>\mkern-7mu\raisebox{-0.065ex}{\rule[0.5865ex]{1.38ex}{0.1ex}}}\,}\xspace}
\newcommand{\XTOP}{\hat{\top}}
\newcommand{\XBOT}{\ensuremath{\check{\bot}}\xspace}
\newcommand{\XAND}{\ensuremath{\:\hat{\wedge}\:}\xspace}
\newcommand{\XOR}{\ensuremath{\:\check{\vee}\:}\xspace}
\newcommand{\XRARR}{\ensuremath{\:\check{\rightarrow}\:}\xspace}
\newcommand{\XCRARR}{\ensuremath{\hat{{\:{>\mkern-7mu\raisebox{-0.065ex}{\rule[0.5865ex]{1.38ex}{0.1ex}}}\:}}}\xspace}

\newcommand{\dneg}{\ensuremath{\sim}}
\newcommand{\rdneg}{\ensuremath{\sim}}
\newcommand{\dtop}{\ensuremath{1}\xspace}
\newcommand{\dbot}{\ensuremath{0}\xspace}
\newcommand{\dand}{\ensuremath{\cap}\xspace}
\newcommand{\dor}{\ensuremath{\cup}\xspace}
\newcommand{\drarr}{\ensuremath{\,{\raisebox{-0.065ex}{\rule[0.5865ex]{1.38ex}{0.1ex}}\mkern-5mu\supset}\,}\xspace}
\newcommand{\dcrarr}{\ensuremath{\,{\supset\mkern-5.5mu\raisebox{-0.065ex}{\rule[0.5865ex]{1.38ex}{0.1ex}}}\,}\xspace}
\newcommand{\DNEG}{\ensuremath{\tilde{*}}\,}

\newcommand{\DTOP}{\ensuremath{\hat{1}}\xspace}
\newcommand{\DBOT}{\ensuremath{\check{0}}\xspace}
\newcommand{\DAND}{\ensuremath{\:\hat{\cap}\:}\xspace}
\newcommand{\DOR}{\ensuremath{\:\check{\cup}\:}\xspace}
\newcommand{\DRARR}{\ensuremath{\check{\,{\raisebox{-0.065ex}{\rule[0.5865ex]{1.38ex}{0.1ex}}\mkern-5mu\supset}\,}}\xspace}

\newcommand{\DCRARR}{\ensuremath{\hat{\,{\supset\mkern-5.5mu\raisebox{-0.065ex}{\rule[0.5865ex]{1.38ex}{0.1ex}}}\,}}\xspace}

\newcommand{\wbox}{\ensuremath{{\Box}}\,}
\newcommand{\wdia}{\ensuremath{{\circ}}\,}

\newcommand{\bbox}{\ensuremath{{\circ}}\,}
\newcommand{\bdia}{\ensuremath{\blacklozenge}\,}
\newcommand{\WBOX}{\ensuremath{\check{\Box}}\,}

\newcommand{\BDIA}{\ensuremath{\:\hat{\blacklozenge}}\,}
\newcommand{\WCIR}{\ensuremath{\tilde{{\circ}}}\,}
\newcommand{\BCIR}{\ensuremath{\:\tilde{{\bullet}}}}

\newcommand{\lh}{\ensuremath{\:\bullet_\ell}\xspace}
\newcommand{\rh}{\ensuremath{\:\bullet_r}\xspace}

\newcommand{\LH}{\ensuremath{\:\hat{\bullet}_\ell}\,}
\newcommand{\RH}{\ensuremath{\:\check{\bullet}_r}\,}

\usetikzlibrary{arrows}
\usetikzlibrary{matrix}
\usetikzlibrary{patterns}
\usetikzlibrary{shapes}
\usepackage{authblk}

\theoremstyle{definition}
\theoremstyle{plain}
\newtheorem{theorem}{Theorem}
\newtheorem{lemma}[theorem]{Lemma}
\newtheorem{corollary}[theorem]{Corollary}
\newtheorem{proposition}[theorem]{Proposition}

\theoremstyle{definition}
\newtheorem{definition}[theorem]{Definition}

\makeindex

\title{Semi De Morgan logic properly displayed}

\author[2]{Giuseppe Greco}
\author[1]{Fei Liang}
\author[3]{M.~Andrew Moshier}
\author[4,5]{ Alessandra Palmigiano\thanks{This research is supported by the NWO Vidi grant 016.138.314, the NWO Aspasia grant 015.008.054, and a Delft Technology Fellowship awarded to the second author in 2013.}}
\affil[1]{Shandong University, China}
\affil[2]{Utrecht University, the Netherlands}
\affil[3]{Chapman University, USA}
\affil[4]{Delft University of Technology, the Netherlands}
\affil[5]{University of Johannesburg, South Africa}

\date{}

\begin{document}
\maketitle

\begin{abstract}

In the present paper, we endow a family of axiomatic extensions of semi De Morgan logic with proper multi-type display calculi which are sound, complete, conservative, and enjoy cut elimination and subformula property. Our proposal builds on an algebraic analysis of semi De Morgan algebras and its subvarieties and applies the guidelines of the multi-type methodology in the design of display calculi.\\
\\
$\mathbf{Keywords:}$ semi De Morgan algebras, proper display calculus, multi-type methodology.\\
$\mathbf{Math.\ Subject\ Class\ 2010:}$  03B45, 03G25, 03F05, 06D30, 08A68.
\end{abstract}

\section{Introduction}\label{sec:Introduction}

Semi De Morgan logic, introduced in an algebraic setting by H.P.~Sankappanavar \cite{sankappanavar1987semi}, is a very well known  paraconsistent logic \cite{priest2002paraconsistent}, and is designed to capture the salient features of intuitionistic negation in a paraconsistent setting. Semi De Morgan algebras form a variety of normal distributive lattice expansions (cf.~\cite[Definition 9]{greco2016unified}), and are a common abstraction of De Morgan algebras and distributive pseudocomplemented lattices. Besides being studied from a universal-algebraic perspective \cite{sankappanavar1987semi, celani1999distributive,celani2007representation}, semi De Morgan logic has been studied from a duality-theoretic perspective \cite{hobby1996semi} and from the perspective of canonical extensions \cite{Palma2005}. 

From a proof-theoretic perspective, the main challenge posed by semi De Morgan logic is that, unlike De Morgan logic,  its axiomatization is not analytic inductive in the sense of \cite[Definition 55]{greco2016unified}. In \cite{greco2017multi}, an analytic calculus for semi De Morgan logic is introduced  which is sound, complete, conservative, and enjoys cut elimination and subformula property. The design of this calculus builds on an algebraic analysis of semi De Morgan algebras, and applies the guidelines of the multi-type methodology, introduced in \cite{Multitype,PDL} and further developed in \cite{BGPTW,TrendsXIII,inquisitive,linearlogPdisplayed,GrecoPalmigianoLatticeLogic}. This methodology guarantees in particular that all the properties mentioned above follow from the general background theory of proper multi-type display calculi (cf.~\cite[Definition A.1.]{linearlogPdisplayed}).

Due to space constraints, in \cite{greco2017multi}, the proofs of the algebraic analysis on which the design of this calculus is grounded had to be omitted. The present chapter provides the missing proofs, and also extends the results of \cite{greco2017multi} by explicitly and modularly accounting for the logics associated with the five subvarieties of semi De Morgan algebras introduced in \cite{sankappanavar1987semi}. This modular account is partly made possible by the fact that all but two of these subvarieties correspond to axiomatic extensions of  semi De Morgan logic with so-called {\em analytic inductive} axioms (cf.~\cite[Definition 55]{greco2016unified}), and the two remaining ones can be given analytic equivalent presentations in the multi-type setting for the basic calculus. The general theory of proper (multi-type) display calculi provides an algorithm which computes the analytic structural rules corresponding to these axioms, and guarantees that each calculus obtained by adding any subset of these rules to the basic calculus still enjoys cut elimination and subformula property.

Therefore,  this chapter introduces  a proof-theoretic environment which is suitable to complement, from a proof-theoretic perspective, the investigations on the lattice of axiomatic extensions of semi De Morgan logic, 
as well as on the connections between the lattices of axiomatic extensions of semi De Morgan logic and of De Morgan logic.

\paragraph{Structure of the chapter.} In Section \ref{sec:pre}, we report on the axioms and rules of semi De Morgan logic and its axiomatic extensions arising from the subvarieties of semi De Morgan algebras introduced in \cite{sankappanavar1987semi}, and discuss why the basic axiomatization is not amenable to the standard treatment of display calculi. In Section \ref{sec:multitype presentation}, we define the algebraic environment which motivates our
multi-type approach and prove that this environment is an equivalent presentation of the standard algebraic semantics of semi De Morgan logic and its extensions. Then we introduce the multi-type semantic environment and define translations between the single-type and the multi-type languages of
semi De Morgan logic and its extensions. In Section \ref{sec: multi-type language}, we discuss how equivalent analytic (multi-type) reformulations can be given of non-analytic (single-type) axioms in the language of semi De Morgan logic.  In Section \ref{sec:proper display}, we introduce the display calculi for semi De Morgan logic and its extensions, and in Section \ref{sec:properties}, we discuss their soundness,  completeness, conservativity, cut elimination and subformula property.

\section{Preliminaries}\label{sec:pre}

\subsection{Semi De Morgan logic and its axiomatic extensions} \label{ssec:Hilbert system}
Fix a denumerable set $\mathsf{Atprop}$ of propositional variables, let $p$ denote an element in $\mathsf{Atprop}$. The language $\mathcal{L}$ of
semi De Morgan logic  over  $\mathsf{Atprop}$ is defined recursively as follows:
$$A\,\, \textrm{::=}\,\, p \mid \top \mid \bot \mid \neg A \mid A \wedge A \mid A \vee A$$

\begin{definition}\label{def:logics}
Semi De Morgan logic, denoted $\msf{SM}$, consists of the following axioms:
{\small
$$\bot \fCenter A,\quad A \fCenter \top,\quad\neg\top \fCenter \bot,\quad \top \fCenter \neg\bot,\quad A \fCenter A,\quad A \wedge B \fCenter A,\quad A\wedge B \fCenter B,$$
$$A \fCenter A \vee B, \quad B \fCenter A \vee B,\quad \neg A \fCenter \neg\neg\neg A,\quad\neg\neg\neg A \fCenter \neg A,\quad \neg A \wedge \neg B \fCenter \neg (A \vee B),$$
$$\neg\neg A \wedge \neg\neg B \fCenter \neg\neg(A \wedge B), \quad A \wedge (B \vee C) \fCenter (A \wedge B) \vee (A \wedge C)$$
}
and the following rules:
\\

\setlength{\parindent}{-1em}
{\small
\begin{tabular}{llll}
\AX $A \fCenter B$
\AX $B \fCenter C$
\BI $A \fCenter C$
\DP
&
\AX $A \fCenter B$
\AX $A \fCenter C$
\BI $A \fCenter B \wedge C$
\DP
&
\AX $A \fCenter B$
\AX $C \fCenter B$
\BI$A \vee C \fCenter B$
\DP
&
\AX$A \fCenter B$
\UI$ {\neg} B \fCenter {\neg} A$
\DP
\end{tabular}
}
\\

The following table reports the name of each axiomatic extension of $\msf{SM}$ arising from the subvarieties introduced in \cite{sankappanavar1987semi}, its acronym, and its characterizing axiom:

\begin{center}
\begin{tabular}{|rcl|}
\hline
lower quasi De Morgan logic&$\msf{LQM}$&$A \vdash \neg\neg A$\\
\hline
upper quasi De Morgan logic&$\msf{UQM}$&$\neg\neg A \vdash A$\\
\hline
demi pseudo-complemented lattice logic&$\msf{DP}$&$\neg A \wedge \neg\neg A \fCenter \bot$\\
\hline
almost pseudo-complemented lattice logic&$\msf{AP}$&$ A \wedge \neg A \fCenter \bot$\\
\hline
weak Stone logic&$\msf{WS}$&$\top \fCenter \neg A \vee \neg\neg A $\\
\hline
\end{tabular}
\end{center}
\end{definition}

In \cite{greco2016unified}, a characterization is given of the properly displayable (single-type) logics (i.e.~those logics that can be captured by a proper display calculus, cf.~\cite[Chapter 4]{Wa98}). Properly displayable logics  are exactly those logics  which admit a presentation consisting  of analytic inductive axioms  (cf.~\cite[Definition 55]{greco2016unified}). It is not difficult to verify that $\neg A \fCenter \neg\neg\neg A,~\neg\neg\neg A \fCenter \neg A$ and $\neg\neg A \wedge \neg\neg B \fCenter \neg\neg(A \wedge B)$  in $\msf{SM}$, $\neg\neg A \fCenter A$ in $\msf{UQM}$,   and $\neg A \wedge \neg\neg A \fCenter \bot$ in $\msf{DP}$ are not analytic inductive. 
To our knowledge, no equivalent axiomatizations have been introduced for semi De Morgan logic and its extensions using only analytic inductive axioms. This provides the motivation for circumventing this difficulty by introducing  proper multi-type display calculi for semi De Morgan logic and its extensions.

\subsection{The variety of semi De Morgan algebras and its subvarieties}\label{sec:sdm+}
We recall the definition of the variety  of semi De Morgan algebras and those of its subvarieties introduced in \cite[Definition 2.2, Definition 2.6]{sankappanavar1987semi}.

\begin{definition}\label{def:sdm}
An algebra $\mathbb{A}$ = $(L, \wedge, \vee, ',  \top, \bot)$ is a {\em semi De Morgan algebra} (SMA) if  for all $a, b \in L$,
\begin{enumerate}
\item[$(\mathrm{S1})$] $(L, \wedge, \vee, 1, 0)$ is a bounded distributive lattice;
\item[$(\mathrm{S2})$] $\bot'$ = $\top, \top'$ = $\bot$;
\item[$(\mathrm{S3})$] $(a \vee b)' $= $a' \wedge b'$;
\item[$(\mathrm{S4})$] $(a \wedge b)''$ = $a'' \wedge b''$;
\item[$(\mathrm{S5})$] $a'$ = $a'''$.
\end{enumerate}
A {\em lower quasi De Morgan algebra} (LQMA) is an SMA satisfying the following inequality:
\begin{enumerate}
\item[$\mathrm{(S6a)}$] $a \leq a''.$
\end{enumerate}
Dually, a {\em upper quasi De Morgan algebra} (UQMA) is an SMA satisfying the following inequality:
\begin{enumerate}
\item[$\mathrm{(S6b)}$] $a'' \leq a.$
\end{enumerate}
A {\em demi pseudocomplemented lattice} (DPL) is an SMA satisfying the following equation:
 \begin{enumerate}
\item[$\mathrm{(S7)}$] $a' \wedge a'' = \bot.$
\end{enumerate}
 A {\em almost pseudocomplemented lattice} (APL) is an SMA satisfying the following equation:
 \begin{enumerate}
\item[$\mathrm{(S8)}$] $a \wedge a' = \bot.$
\end{enumerate}
 A {\em weak Stone algebra} (WSA) is an SMA satisfying the following equation:
 \begin{enumerate}
\item[$\mathrm{(S9)}$] $a' \vee a'' = \top.$
\end{enumerate}
  \end{definition}
 The following proposition is a straightforward consequence of (S8), (S2), (S3) and (S5):
 \begin{proposition}[\cite{sankappanavar1987semi} see discussion above Corollary 2.7]\label{fact: s7}
$\mathrm{(S7)}$ holds in any APL and WSA.
\end{proposition}

\begin{definition}\label{def:dm}
An algebra $\mathbb{D} = (D, \cap, \cup, ^*,  1, 0)$ is a {\em De Morgan algebra} (DMA) if   for all $a, b \in D$,
\begin{enumerate}
\item[$(\mathrm{D1})$] $(D, \cap, \cup, 1, 0)$ is a bounded distributive lattice;
\item[$(\mathrm{D2})$] $0^*$ = $1, 1^*$ = $0$;
\item[$(\mathrm{D3})$] $(a \cup b)^*$ = $a^* \cap b^*$;
\item[$(\mathrm{D4})$] $(a \cap b)^*$ = $a^* \cup b^*$;
\item[$(\mathrm{D5})$] $a$ = $a^{**}$.
\end{enumerate}
\end{definition}
\noindent As is well known, a Boolean algebra (BA) $\mathbb{D}$ is a  DMA  satisfying one of the following equations:
\begin{enumerate}
\item[$(\mathrm{B}1)$] $a \vee a^* = 1$;
\item[$(\mathrm{B}2)$] $a \wedge a^* = 0.$
\end{enumerate}
The following theorem can be shown using a routine Lindenbaum-Tarski construction.
\begin{theorem}[\textbf{Completeness}]\label{completeness: H.SDM}
$\msf{SM}$ (resp.~$\msf{LQM}$, $\msf{UQM}$, $\msf{DP}$, $\msf{AP}$, $\msf{WS}$) is complete with respect to the class of SMAs (resp.~LQMAs, UQMAs, DPLs, APLs, WSAs).
\end{theorem}

\begin{definition}\label{def: perfect algebra}
A distributive lattice $\mathbb{A}$ is  {\em perfect} (cf.~\cite[Definition 2.14]{GNV05}) if
$\mathbb{A}$ is complete, completely distributive and completely join-generated by the set $J^{\infty}(\mathbb{A})$ of its completely join-irreducible elements (as well as completely meet-generated by the set  $M^{\infty}(\mathbb{A})$ of its completely meet-irreducible elements).

A De Morgan algebra (resp.~Boolean algebra) $\mathbb{A}$ is {\em perfect} if its lattice reduct is a perfect distributive lattice, and the following distributive laws are valid:
\begin{center}
$(\bigvee X)^{*} \,\,\textrm{=}\,\, \bigwedge X^{*} \quad\quad (\bigwedge X)^{*}\,\,\textrm{=}\,\, \bigvee X^{*}.$
\end{center}
A lattice homomorphism $h: \mathbb{L} \rightarrow  \mathbb{L'}$ is  {\em complete} if  for each $X\subseteq \mathbb{L}$,
\begin{center}
$h(\bigvee X) \,\,\textrm{=}\,\,\bigvee h(X) \quad\quad h(\bigwedge X) \,\,\textrm{=}\,\, \bigwedge h(X).$
\end{center}
\end{definition}

\section{Towards a multi-type presentation}\label{sec:multitype presentation}
In the present section, we introduce the algebraic environment which justifies semantically the multi-type approach to semi De Morgan logic and its extensions of Section \ref{ssec:Hilbert system}. In the next subsection, we define the kernel of an SMA (cf.~Definition \ref{def:kernel}) and show that it can be endowed with a structure of DMA (cf.~Definition \ref{def:dm}). Similarly, we define the kernel of a DPL (cf.~Definition \ref{def:sdm}) and show that it can be endowed with a structure of Boolean algebra. Then we define two maps between the kernel of any SMA (resp.~DPL) $\mathbb{A}$ and the lattice reduct of $\mathbb{A}$. These are the main components of the heterogeneous semi De Morgan algebras and the heterogeneous demi p-lattices which we introduce in Subsection \ref{Heterogeneous presentation}, where we also show that SMAs (resp.~DPLs) can be equivalently presented in terms of heterogeneous semi De Morgan algebras (heterogeneous demi p-lattices). Based on these, we can also define the heterogeneous  algebras for other subvariety of SMAs we introduced in  Section \ref{sec:sdm+}. In Subsection \ref{Canonical extensions}, we apply results pertaining to the theory of canonical extensions to the heterogeneous  semi De Morgan algebras and the heterogeneous demi p-lattices.

\subsection{The kernel of a semi De Morgan algebra}\label{ssec:kernel}
For any semi De Morgan algebra $\mathbb{A} = (L, \wedge, \vee, ',  \top, \bot)$, we let $K:=\{a'' \mid a \in L\}$,  define $h: L \twoheadrightarrow K$ by the assignment  $a\mapsto a''$ for any $a \in L$, and let $e: K \hookrightarrow L$ denote the natural embedding. Hence,  $eh(a) = a''$ and $h(a) = h(a'')$ for every $a\in L$. \begin{lemma}
\label{eq:retractions}
For any  semi De Morgan algebra $\mathbb{A}$, and $K, h, e$ defined as above, the following equation holds for any $\alpha \in K$:
\begin{equation}\label{eq:composition}
he(\alpha) \,\,\textrm{=}\,\, \alpha
\end{equation}
\end{lemma}
\begin{proof}
Let $\alpha \in K$, and let $a\in L$ such that $h(a) = \alpha$. Hence,
\begin{center}
\begin{tabular}{rlll}
$he(\alpha)$ & = &$heh(a)$& $\alpha = h(a)$\\
&= &$h(a'')$&  $eh(a) = a''$\\
& = &$h(a)$& $h(a) = h(a'')$\\
&= &$\alpha$&definition  of $h$
\end{tabular}
\end{center}
\end{proof}

\begin{definition}
\label{def:kernel}
For any SMA $\mathbb{A} = (L, \wedge, \vee, \top, \bot, ')$,  let the {\em kernel} $\mathbb{K}_{\mathbb{A}} = (K, \cap, \cup, ^*,  1, 0)$
 of $\mathbb{A}$ be   defined as follows:
\begin{itemize}
\item[$\mathrm{K1}$] $K: = \{a'' \mid a \in L \}$;
\item[$\mathrm{K2}$] $\alpha \cup \beta:$ = $h((e (\alpha) \vee e(\beta))'')$ for all $\alpha, \beta\in K$;
\item[$\mathrm{K3}$] $\alpha \cap \beta:$ = $h (e(\alpha) \wedge e(\beta))$ for all $\alpha, \beta\in K$;
\item[$\mathrm{K4}$]  $1:$  = $h (\top)$;
\item[$\mathrm{K5}$]  $0:$ = $h(\bot)$;
\item[$\mathrm{K6}$] $\alpha^*: $= $h(e(\alpha)')$.
\end{itemize}
\end{definition}

In what follows, to simplify the notation, we omit  as many parentheses as we can without generating ambiguous readings. For example, we write $e(h(a)^*)$ in place of $e((h(a))^*)$, and $eh(a)'$ in place of $(eh(a))'$.
\begin{proposition}\label{prop: algebraic structure on kernels}
If $\mathbb{A} = (L, \wedge, \vee, \top, \bot, ')$ is an  SMA,  then $\mathbb{K}_\mathbb{A}$ is a De Morgan algebra.
\end{proposition}

\begin{proof}
Let us show that  $\mathbb{K}_\mathbb{A}$ is a distributive lattice. Associativity and commutativity are straightforwardly verified and their corresponding verification is omitted. To show that the absorption law and the distributive law hold, let $\alpha, \beta, \gamma \in K$, and let $a, b, c \in L$ such that (i) $h(a) = \alpha$, (ii) $h(b) = \beta$ and (iii) $h(a) = \gamma$.
\begin{itemize}
\item absorption law:
\begin{center}
\begin{tabular}{lll}
&$\alpha \cup (\alpha \cap \beta)$&\\
=&$h((e(\alpha) \vee e(\alpha \cap \beta))'')$& K2\\
=&$h((e(\alpha) \vee eh(e(\alpha) \wedge e(\beta)))'')$& K3\\
=&$h((e(\alpha) \vee (e(\alpha) \wedge e(\beta))'')'')$&  $eh(a) = a''$\\
=&$h((e(\alpha)' \wedge (e(\alpha) \wedge e(\beta))''')')$& S3\\
=&$h((e(\alpha)''' \wedge (e(\alpha)'' \wedge e(\beta)'')')')$&S5, S4\\
=&$h((e(\alpha)'' \vee (e(\alpha)'' \wedge e(\beta)''))'')$&S3\\
=&$h((e(\alpha)'' \vee (e(\alpha)'''' \wedge e(\beta)''''))$&S4\\
=&$h((e(\alpha)'' \vee (e(\alpha)'' \wedge e(\beta)''))$&S5\\
=&$h(e(\alpha)'')$&S1\\
=&$he(\alpha)$& $h(a) = h(a'')$\\
= &$\alpha$&Lemma \ref{eq:retractions}\\
\end{tabular}
\end{center}

\item distributivity law:
\begin{center}
\begin{tabular}{lll}
&$\alpha \cap (\beta \cup \gamma)$&\\
=&$h(e(\alpha) \wedge e(\beta \cup \gamma))$& K3\\
=&$h(e(\alpha) \wedge eh((e(\beta) \vee e(\gamma))''))$& K2\\
= &$h(e(\alpha) \wedge (e(\beta) \vee e(\gamma))'''')$& $eh(a) = a''$\\
= &$h(eh(a) \wedge (e(\beta) \vee e(\gamma))'''')$& (i)\\
= &$h(a'' \wedge (e(\beta) \vee e(\gamma))'''')$& $eh(a) = a''$\\
= &$h(a'''' \wedge (e(\beta) \vee e(\gamma))'')$&S5\\
= &$h((a'' \wedge (e(\beta) \vee e(\gamma)))'')$&S4\\
= &$h(((a'' \wedge e(\beta)) \vee (a'' \wedge e(\gamma)))'')$&S1\\
= &$h((a'' \wedge eh(b)) \vee (a'' \wedge eh(c))'')$&(ii) and (iii)\\
= &$h(((a'' \wedge b'') \vee (a'' \wedge c''))'')$&$eh(a) = a''$\\
= &$h(((a'''' \wedge b'''') \vee (a'''' \wedge c''''))'')$&S5\\
= &$h(((a'' \wedge b'')'' \vee (a'' \wedge c'')'')'')$&S4\\
= &$h((eh(eh(a) \wedge eh(b)) \vee eh(eh(a) \wedge eh(c)))'')$&$eh(a) = a''$\\
= &$h((eh(e(\alpha) \wedge e(\beta)) \vee eh(e(\alpha) \wedge e(\gamma)))'')$&(i), (ii) and (iii)\\
= &$h(((e(\alpha \cap \beta)) \vee e(\alpha \cap \gamma))'')$&K3\\
= &$(\alpha \cap \beta) \cup (\alpha \cap \gamma)$&K2\\
\end{tabular}
\end{center}
\end{itemize}

Let us show that  $\mathbb{K}_\mathbb{A}$ satisfies (D1)-(D5).

As to (D1), we need to show that $\mathbb{K}_\mathbb{A}$ is bounded:
\begin{center}
\begin{tabular}{lllclll}
&$0 \cap \alpha$&&$\quad$&&$1 \cup \alpha$&\\
=&$h(e(0) \wedge e(\alpha))$& K3&$\quad$&=&$h((e(1) \vee e(\alpha))'')$& K2\\
=&$h(eh(\bot) \wedge e(\alpha))$&K5&$\quad$&=&$h((eh(\top)) \vee e(\alpha))'')$& K4\\
=&$h(\bot'' \wedge e(\alpha))$& $eh(a) = a''$&$\quad$&=&$h((\top'' \vee e(\alpha))'')$& $eh(a) = a''$\\
=&$h(\bot \wedge e(\alpha))$& S2&$\quad$&=&$h((\top \vee e(\alpha))'')$& S2\\
=&$h(\bot)$& S1&$\quad$&=&$h(\top'')$& S5\\
=&$0$& K5&$\quad$&=&$1$& S2, K4\\
\end{tabular}
\end{center}

As to (D2):
\begin{center}
\begin{tabular}{llllllll}
$0^*$& = &$h(e(0)')$&K6&$1^*$&=&$h(e(1)')$& K6\\
&=&$h((eh(\bot))')$& K5&&=&$h((eh(\top))')$& K4\\
&=&$h(\bot''')$& $eh(a) = a''$&&=&$h(\top''')$&$eh(a) = a''$\\
&=&$h(\bot')$& S5&&=&$h(\top')$&S5\\
&=&$h(\top)$& S2&&=&$h(\bot)$& S2\\
&=&$ 1$& K4&&=&$ 0$& K4\\
\end{tabular}
\end{center}

As to (D3):
\begin{center}
\begin{tabular}{llll}
$(\alpha \cup \beta)^*$&=&$h(e(\alpha \cup \beta)')$&K6\\
&= &$h((eh((e(\alpha) \vee e(\beta))'')')$&K2\\
&=&$h((e(\alpha) \vee e(\beta))''''')$&$eh(a) = a''$\\
&=&$h((e(\alpha)' \wedge e(\beta)')'''')$& S3\\
&=&$h((e(\alpha)''' \wedge e(\beta)''')'')$& S5\\
&=&$h((eh(e(\alpha)') \wedge eh(e(\beta)'))'')$&$eh(a) = a''$\\
&=&$h((e(\alpha^*) \wedge e(\beta^*))'')$& K6\\
&=&$heh(e(\alpha^*) \wedge e(\beta^*))$&$eh(a) = a''$\\
&=&$h(e(\alpha^*) \wedge e(\beta^*))$& Lemma \ref{eq:retractions}\\
&=&$\alpha^* \cap \beta^*$& K3\\
\end{tabular}
\end{center}

As to (D4):
\begin{center}
\begin{tabular}{llll}
$(\alpha \cap \beta)^*$&=&$h(e(\alpha \cap \beta)')$& K6\\
&=&$h((eh(e(\alpha) \wedge e(\beta)))')$& K3\\
&=&$h((e(\alpha) \wedge e(\beta))''')$&$eh(a) = a''$ \\
&=&$h((e(\alpha)'' \wedge e(\beta)'')')$& S4\\
&=&$h((e(\alpha)' \vee e(\beta)')'')$& S3\\
&=&$h((eh(a)' \vee eh(b)')'')$& (i) and (ii)\\
&=&$h((a''' \vee b''')'')$& $eh(a) = a''$\\
&=&$h((a''''' \vee b''''')'')$& S5\\
&=&$h((eh(eh(a)') \vee eh(eh(b)'))'')$& $eh(a) = a''$\\
&=&$h((eh(e(\alpha)') \vee eh(e(\beta)'))'')$&(i) and (ii) \\
&=&$h((e(\alpha^*) \vee e(\beta^*))'')$& K6\\
&=&$\alpha^* \cup \beta^*$& K2\\
\end{tabular}
\end{center}

As to (D5):
\begin{center}
\begin{tabular}{llll}
$\alpha^{**}$&=&$h((eh(e(\alpha)'))')$& K6\\
&=&$h(e(\alpha)'''')$& $eh(a) = a''$\\
&=&$h((eh(a))'''')$& (i)\\
&=&$h(a'''''')$& $eh(a) = a''$\\
&=&$h(a'')$& S5\\
&=&$heh(a)$& $eh(a) = a''$\\
&=&$h(a)$&  Lemma \ref{eq:retractions}\\
&=&$\alpha$& (i)\\
\end{tabular}
\end{center}

\end{proof}

\begin{corollary}\label{prop: algebraic structure on kernels of dpm}
If $\mathbb{A} = (L, \wedge, \vee, \top, \bot, ')$ is a DPL, then $\mathbb{K}_\mathbb{A}$ is a Boolean algebra.
\end{corollary}

\begin{proof}
By Proposition \ref{prop: algebraic structure on kernels}, $\mathbb{K}_\mathbb{A}$ is a De Morgan algebra. Hence, it suffices to show that $\mathbb{K}$ satisfies (B1).  For any $\alpha \in \mathbb{K}_\mathbb{A}$,

\begin{center}
\begin{tabular}{llll}
$\alpha \cap \alpha^{*}$& = &$h(\alpha \cap h(e(\alpha)'))$& K3\\
&= &$h(e(\alpha) \wedge eh(e(\alpha)'))$& K6\\
&=&$h(e(\alpha) \wedge e(\alpha)''')$& $eh(a) = a''$\\
&=&$h(e(\alpha) \wedge e(\alpha)')$& S5\\
&=&$heh(e(\alpha) \wedge e(\alpha)')$&  Lemma \ref{eq:retractions}\\
&=&$h((e(\alpha) \wedge e(\alpha)')'')$&  $eh(a) = a''$\\
&=&$h(e(\alpha)'' \wedge e(\alpha)''')$&  S4\\
&=&$h(\bot)$&S7\\
&=&$0$&K5\\
\end{tabular}
\end{center}
\end{proof}

\begin{proposition}\label{prop:map}
Let $\mathbb{A}$ be an SMA (resp.~a DPL), and $e, h$ be defined as above.  Then $h$ is a lattice homomorphism from $\mathbb{A}$ onto $\mathbb{K}_\mathbb{A}$, and  for all $\alpha, \beta\in K$,
$$e(\alpha) \wedge  e(\beta) \,\,\textrm{=}\,\,  e(\alpha\cap \beta)  \quad\quad e(1) \,\,\textrm{=}\,\, \top\quad\quad e(0) \,\,\textrm{=}\,\, \bot.$$
\end{proposition}

\begin{proof}
It is an immediate consequence of K1 that $h$ is surjective. We need to show that $h$ is a lattice homomorphism. For any $a, b \in L$,
\begin{center}
\begin{tabular}{lllclll}
&$h(a \wedge b)$&&$\,\,$&&$h(a \vee b)$&\\
=&$heh(a \wedge b)$& Lemma \ref{eq:retractions}&&=&$heh(a \vee b)$&  Lemma \ref{eq:retractions}\\

=&$h((a \wedge b)'')$& $eh(a) = a''$&&=&$h((a \vee b)'') $&$eh(a) = a''$ \\

=&$h(a'' \wedge b'')$&S4&&=& $h((a' \wedge b')')$ & S3\\

=&$h(eh(a) \wedge eh(b) )$& $eh(a) = a''$&&=& $h(a''' \wedge b''')'$ & S5\\

=&$h(a) \cap h(b)$& K3&&=& $h(a'' \vee b'')''$ & S3\\

&&&&=& $h((eh(a) \vee eh(b))'')$ & $eh(a) = a''$ \\
&&&&=& $h(a) \cup h(b)$ & K2\\

\end{tabular}
\end{center}

Moreover, $h(\bot) = \bot'' = \bot$ and  $h(\top) = \top'' = \top$. This completes the proof that $h$ is a homomorphism from $\mathbb{A}$ to $\mathbb{K}_\mathbb{A}$. Next, we show that $e(\alpha) \wedge  e(\beta) =  e(\alpha \cap \beta)$.  For any $\alpha, \beta \in K$,
\begin{center}
\begin{tabular}{llll}
$e(\alpha \cap \beta)$& =&$eh(e(\alpha) \wedge e(\beta))$& K3\\
&=&$(e(\alpha) \wedge e(\beta))''$&$eh(a) = a''$\\
&=&$e(\alpha)'' \wedge e(\beta)''$& S4\\
&=&$ehe(\alpha) \wedge ehe(\beta)$&$eh(a) = a''$\\
&=&$e(\alpha) \wedge e(\beta)$& Lemma \ref{eq:retractions}\\
\end{tabular}
\end{center}
Finally, $e(0) = eh(\bot) = \bot'' = \bot$ and $e(1) = eh(\top) = \top'' = \top$ are straightforward consequences of  (K4), (K5) and (S2).
\end{proof}



In what follows, we will drop the subscript of the kernel whenever it does not cause confusion.

\subsection{Heterogeneous SMAs as equivalent presentations of  SMAs}
\label{Heterogeneous presentation}
\begin{definition}
\label{def:heterogeneous algebras}
A {\em heterogeneous semi De Morgan algebra} (HSMA) is a tuple $(\mathbb{L}, \mathbb{D}, e, h)$ satisfying the following conditions:
\begin{enumerate}
\item[$(\mathrm{H1})$] $\mathbb{L}$ is a bounded distributive lattice;
\item[$(\mathrm{H2a})$] $\mathbb{D}$  is a De Morgan algebra;
\item[$(\mathrm{H3})$] $e: \mathbb{D}\hookrightarrow \mathbb{L}$  is an order embedding, and for all $\alpha_1,\alpha_2 \in \mathbb{D}$,
\begin{itemize}
  \item[-] $e(\alpha_1)\wedge e(\alpha_2)  =  e(\alpha_1 \cap \alpha_2) $;
  \item[-] $e(1) = \top$, $e(0) = \bot$.
\end{itemize}
\item[$(\mathrm{H4})$] $h: \mathbb{L}\twoheadrightarrow \mathbb{D}$  is a surjective lattice homomorphism;
\item[$(\mathrm{H5})$] $he(\alpha)$ = $\alpha$ for every $\alpha\in \mathbb{D}$.\footnote{Condition (H5) implies that $h$ is surjective and $e$ is injective.}
\end{enumerate}

A {\em heterogeneous  lower quasi De Morgan algebra} (HLQMA) is an HSMA satisfying the following condition:
\begin{itemize}
\item[$\mathrm{(H6a)}$] $a \leq eh(a)$ for any $a \in L$.
\end{itemize}

A {\em heterogeneous  upper quasi De Morgan algebra} (HUQMA) is an HSMA satisfying the following condition:
\begin{itemize}
\item[$\mathrm{(H6b)}$] $eh(a) \leq a$ for any $a \in L$.
\end{itemize}
A {\em heterogeneous demi pseudocomplemented lattice} (HDPL) is defined analogously, except replacing (H2a)  with the following condition (H2b):
\begin{enumerate}
\item[$(\mathrm{H2b})$] $\mathbb{D}$  is a Boolean algebra.
\end{enumerate}

A {\em heterogeneous almost pseudocomplemented lattice} (HAPL) is an HDPL satisfying the following condition:
 \begin{itemize}
\item[$(\mathrm{H7})$] $e(h(a)^*) \wedge a = \bot$ for all $a \in \mathbb{L}$.
\end{itemize}
 A {\em heterogeneous weak Stone algebra} (HWSA) is an HDPL satisfying  the following condition:
 \begin{itemize}
\item[$(\mathrm{H8})$] $e(\alpha^*) \vee e(\alpha) = \top$ for all $\alpha \in \mathbb{A}$.
\end{itemize}

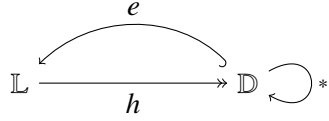
$$
\begin{tikzpicture}[node/.style={circle, draw, fill=black}, scale=1]
\node (DL) at (-1.5,-1.5) {$\mathbb{L}$};
\node (DM) at (1.5, -1.5) {$\mathbb{D}$};
\node (a) at (2.5,-1.58) {$^\ast$};
\draw [->]  (DM)  to [out= 30,in= -30,looseness = 7](DM);
\draw [->>] (DL)  to node[below] {$h$}  (DM);
\draw [right hook->] (DM) to [out=135, in = 45, looseness=1]   node[above] {$e$}  (DL);
\end{tikzpicture}
$$
An HSMA (resp.~HLQMA,HUQMA, HDPL, HAPL and HWSA) is {\em perfect} if:
\begin{enumerate}
\item[$(\mathrm{PH1})$] both $\mathbb{L}$ and $\mathbb{D}$ are perfect as a distributive lattice and De Morgan algebra (or~Boolean algebra), respectively (see Definition \ref{def: perfect algebra});
\item[$(\mathrm{PH2})$] $e$ is an order-embedding and is completely meet-preserving;
\item[$(\mathrm{PH3})$] $h$ is a complete homomorphism.
\end{enumerate}
\end{definition}

\begin{definition}\label{def: aplus}
For any SMA (resp.~LQMA, UQMA, DPL, APL and WSA) $\mathbb{A}$, let $$\mathbb{A}^+:= (\mathbb{L}, \mathbb{K}, e, h),$$
where $\mathbb{L}$ is the lattice reduct of $\mathbb{A}$, $\mathbb{K}$ is the kernel of $\mathbb{A}$ (cf.~Definition \ref{def:kernel}), and $e: \mathbb{K} \to \mathbb{L} $ and $ h: \mathbb{L} \to \mathbb{K}$ are defined as in the beginning of Section \ref{ssec:kernel}.
\end{definition}

\begin{proposition}
\label{prop:from single to multi}
 If $\mathbb{A}$ is an SMA (resp.~DPL), then $\mathbb{A}^+$ is an HSMA (resp.~HDPL).
\end{proposition}

\begin{proof}
It immediately follows from Proposition \ref{prop: algebraic structure on kernels} and Proposition \ref{prop:map}.
\end{proof}

\begin{corollary}\label{cor:sub single to multi}
 If $\mathbb{A}$ is an LQMA (resp.~UQMA, APL and WSA), then $\mathbb{A}^+$ is an HLQMA (resp.~HUQMA, HAPL and HWSA).
\end{corollary}

\begin{proof}
 If $\mathbb{A}$ is an LQMA, by Proposition \ref{prop:from single to multi}, it suffices to show that $\mathbb{A}_+$ satisfies (H6a).  By (S6a) and H5, it is easy to see $a \leq e(h(a))$. The argument is dual when $\mathbb{A}$ is a UQMA. If $\mathbb{A}$ is an APL,  it suffices to show $\mathbb{A}^+$ satisfies (H7).
\begin{center}
\begin{tabular}{lll}
&$e(h(a)^*) \wedge a$&\\
=&$eh((eh(a))') \wedge a$& K6\\
=&$a^{'''''} \wedge a$& $eh(a) = a''$\\
=&$a' \wedge a$& S5\\
=&$\bot$& S8\\
\end{tabular}
\end{center}
If $\mathbb{A}$ is a WSA, it suffices to show $\mathbb{A}^+$ satisfies (H8).
\begin{center}
\begin{tabular}{lll}
&$e(\alpha^*) \vee e(\alpha)$&\\
=&$eh(e(\alpha)') \vee e(\alpha)$& K6\\
=&$eh(e(\alpha)') \vee ehe(\alpha)$&Lemma  \ref{eq:retractions} \\
=&$e(\alpha)''' \vee ehe(\alpha)$&$eh(a) = a''$\\
=&$e(\alpha)''' \vee e(\alpha)''$&Lemma  \ref{eq:retractions} \\
=&$e(\alpha)' \vee e(\alpha)''$&S5 \\
=&$\top$&S9\\
\end{tabular}
\end{center}
\end{proof}

\begin{definition}
For any HSMA  (resp.~HLQMA, HUQMA,  HDPL, HAPL and HWSA) $\mathbb{H} = (\mathbb{L}, \mathbb{D}, e, h)$, let $$\mathbb{H}_+:= (\mathbb{L}, \, '),$$
where $': \mathbb{L}\rightarrow \mathbb{L}$  is defined by the assignment $a' \mapsto e(h(a)^*)$.
\end{definition}

\begin{proposition}
\label{prop:reverse engineering}
If  $\mathbb{H}$ is an HSMA (resp.~HDPL), then $\mathbb{H}_+$ is an SMA (resp.~DPL). Moreover, $\mathbb{K}_{\mathbb{H}^+}\cong \mathbb{K}$.
\end{proposition}

\begin{proof}
Since $\mathbb{H}$ is an HSMA by assumption, $\mathbb{L}$ is a bounded distributive lattice, hence it suffices to show that the operation $'$ satisfies (S2)-(S5) (cf.~Definition \ref{def:sdm}).  

\begin{itemize}
\item As to (S2):

\begin{center}
\begin{tabular}{llllllll}
$\bot' $&=&$e(h(\bot)^*)$& definition of $'$ &$\top'$&=&$e(h(\top)^*)$& definition of $'$\\
&=&$e(0^*)$& H3&&=&$e(1^*)$&H3\\
&=&$e(1)$& H2a&&=& $e(0)$ &H2a\\
&=&$\top$& H3&&=& $\bot$ &H3\\
\end{tabular}
\end{center}

\item As to (S3):
\begin{center}
\begin{tabular}{llll}
$(a \vee b)' $&=&$e(h(a \vee b)^*)$& definition of $'$\\
&=&$e((h(a) \cup h(b))^*)$&H4\\
&=&$e(h(a)^* \cap h(b)^*)$&H2a\\
&=&$e(h(a)^*) \wedge e(h(b)^*)$&H3\\
&=&$a' \wedge b'$& definition of $'$\\
\end{tabular}
\end{center}

\item As to (S4):
\begin{center}
\begin{tabular}{llll}
$(a \wedge b)'' $&=&$ e((he(h(a \wedge b)^*))^*)$& definition of $'$\\
&=&$e(h(a \wedge b)^{**})$& H5\\
&=&$eh(a \wedge b)$& H2a\\
&=&$e(h(a) \cap h(b))$& H4\\
&=&$eh(a) \wedge eh(b)$& H3\\
&=&$e(h(a)^{**}) \wedge e(h(b)^{**})$& H2a\\
&=&$e((he(h(a)^*))^*) \wedge e((he(h(b)^*))^*) $& H5\\
&=&$ a'' \wedge b''$&definition of $'$\\
\end{tabular}
\end{center}

\item As to (S5):
\begin{center}
\begin{tabular}{llll}
$a''' $&=&$e((he((he(h(a)^*))^*))^*)$&definition of $'$\\
&=&$e(h(a)^{***}) $& H5\\
&=&$e(h(a)^*)$& H2a\\
&=&$a'$& definition of $'$\\
\end{tabular}
\end{center}
\end{itemize}

Hence, $(\mathbb{L}, ')$ is a semi De Morgan algebra. If  $(\mathbb{L}, \mathbb{D}, e, h)$ is an HDPL, we  also need to show that $'$ satisfies (S7):
\begin{center}
\begin{tabular}{llll}
$a' \wedge a'' $&= &$e(h(a)^*) \wedge e((he(h(a)^*))^*)$& definition of $'$\\
&= &$e(h(a)^*) \wedge e(h(a)^{**})$&H5\\
&= &$e(h(a)^*) \wedge eh(a)$& H2a\\
&= &$e(h(a)^* \cap h(a))$& H3\\
&= &$e(0)$& H2a\\
&= &$\bot$& H3\\
\end{tabular}
\end{center}
which completes the proof that (L, $'$) is a DPL.  As to the second part of the statement, let us show preliminarily that the following identities hold:

\begin{itemize}
\item[$\mathrm{K2}_{\mathbb{D}}$.] $\alpha \cup \beta = h((e (\alpha) \vee e(\beta))'')$ for all $\alpha, \beta\in \mathbb{D}$;
\item[$\mathrm{K3}_{\mathbb{D}}$.] $\alpha \cap \beta = h (e(\alpha) \wedge e(\beta))$ for all $\alpha, \beta\in \mathbb{D}$;
\item[$\mathrm{K4}_{\mathbb{D}}$.]  $1 = h (\top)$;
\item[$\mathrm{K5}_{\mathbb{D}}$.]  $0 = h(\bot)$;
\item[$\mathrm{K6}_{\mathbb{D}}$.] $\alpha^*  = h(e(\alpha)')$.
\end{itemize}

As to $\mathrm{K2}_{\mathbb{D}}$,
\begin{center}
\begin{tabular}{llll}
$h((e (\alpha) \vee e(\beta))'')$&=&$he((he(h(e (\alpha) \vee e(\beta))^*))^*)$&definition of $'$\\
&=&$(h(e (\alpha) \vee e(\beta)))^{**}$&  H5\\
&=&$h(e(\alpha) \vee e(\beta))$& H2a\\
&=&$he(\alpha) \cup he(\beta)$&  H4\\
&=&$\alpha \cup \beta$ &H5 \\
\end{tabular}
\end{center}
Conditions  $\mathrm{K3}_{\mathbb{D}}$,$\mathrm{K4}_{\mathbb{D}}$ and $\mathrm{K5}_{\mathbb{D}}$ easily follow from H4, H5 and H3, and their proofs are omitted.

As to $\mathrm{K6}_{\mathbb{D}}$,
\begin{center}
\begin{tabular}{llll}
$h(e(\alpha)')$&=&$he((he(\alpha))^*)$& definition of $ '$\\
&=&$\alpha^*$&  H5\\
\end{tabular}
\end{center}

To show that $\mathbb{D}$ and $\mathbb{K}$ are isomorphic to each other, notice that the domain of $\mathbb{K}$ is defined as $K :=  \mathsf{Range}('') = \mathsf{Range}(e\circ ^* \circ h \circ e\circ ^* \circ h) = \mathsf{Range}(e\circ h)$. Since by assumption $h$ is surjective, $K = \mathsf{Range}(e)$, and since $e$ is an order embedding, $K$, regarded as a sub-poset of $\mathbb{L}$, is order-isomorphic to the domain of  $\mathbb{D}$ with its lattice order. Let $f: \mathbb{D} \rightarrow  \mathbb{K}$ denote the order-isomorphism between $\mathbb{D}$ and $\mathbb{K}$. Define $e_k: \mathbb{K} \hookrightarrow \mathbb{L}$ and $h_k: \mathbb{L} \twoheadrightarrow \mathbb{K}$ as as in the beginning of Section \ref{ssec:kernel}. Thus, $e = e_k \circ f$ and $h_k = f \circ h$. We need to show that: for all $\alpha,\beta \in \mathbb{D}$, let $\cap_k, \cup_k,\, ^{\ast_k}$ denote the operations on K,

\begin{enumerate}
\item $f(\alpha) \cap_k f(\beta) = f(\alpha \cap \beta)$,
\begin{center}
\begin{tabular}{llll}
$f(\alpha) \cap_k f(\beta)$&=&$h_k(e_k f(\alpha) \wedge e_kf(\beta)) $& definition of $\cap_k$\\
&=&$fh(e_kf(\alpha) \wedge e_kf(\beta)) $& $h_k = f \circ h$\\
&=&$fh(e(\alpha) \wedge e(\beta)) $& $e = e_k \circ f$\\
&=&$f(\alpha \cap \beta) $& $\mathrm{K3}_{\mathbb{D}}$\\
\end{tabular}
\end{center}

\item $f(\alpha) \cup_k f(\beta) = f(\alpha \cap \beta)$,
\begin{center}
\begin{tabular}{llll}
$f(\alpha) \cup_k f(\beta)$&=&$h_k((e_kf(\alpha) \vee e_kf(\beta))'') $& definition of $\cup_k$\\
&=&$fh((e_kf(\alpha) \vee e_kf(\beta))'')$&$h_k = f \circ h$\\
&=&$fh((e(\alpha)\vee e(\beta))'') $& $e = e_k \circ f$\\
&=&$f(\alpha \cup \beta) $& $\mathrm{K2}_{\mathbb{D}}$\\
\end{tabular}
\end{center}

\item $f(\alpha)^{*_k} = f(\alpha^*)$,
\begin{center}
\begin{tabular}{llll}
$(f(\alpha))^{*_k}$&=&$h_k((e_kf(\alpha))') $& definition of $^{\ast_k}$\\
&=&$fh((e_kf(\alpha))') $&$h_k = f \circ h$\\
&=&$fh(e(\alpha)')  $& $e = e_k \circ f$\\
&=&$f(\alpha^*) $& $\mathrm{K6}_{\mathbb{D}}$\\
\end{tabular}
\end{center}
\end{enumerate}

Hence, $f: \mathbb{D} \rightarrow \mathbb{K}$ is an isomorphism of De Morgan algebras (resp.~Boolean algebras). This completes the proof.
\end{proof}

\begin{corollary}\label{cor:sub multi to single}
 If $\mathbb{H}$ is an HLQMA (resp.~HUQMA, HAPL and HWSA), then $\mathbb{A}^+$ is a LQMA (resp.~UQMA, APL and WSA).
\end{corollary}

\begin{proof}
By Proposition \ref{prop:reverse engineering}, if $\mathbb{H}$ is an HLQMA,  it suffices to show that  $\mathbb{H}_+$ satisfies (S6a).
\begin{center}
\begin{tabular}{lll}
&$a \leq eh(a) $& H6a\\
iff &$a \leq e(h(a)^{**})$&H2a\\
iff &$a \leq e((he(h(a)^{*}))^*)$&H5\\
iff &$a \leq a''$& definition of $'$\\
\end{tabular}
\end{center}
If $\mathbb{H}$ is an HUQMA, the argument is dual.  If $\mathbb{H}$ is an HAPL, it is clear that $\mathbb{H}_+$ satisfies (S8) by (H7).  If $\mathbb{H}$ is an HWSA, it suffices to show that $\mathbb{H}_+$ satisfies (S6).
\begin{center}
\begin{tabular}{lll}
&$a' \vee a'' $&\\
= &$e(h(a)^*) \vee e((he(h(a)^{*}))^*)$& def. of $'$\\
= &$e(h(a)^*) \vee e(h(a)^{**})$&Lemma  \ref{eq:retractions} \\
= &$e(h(a)^*) \vee eh(a)$& H2a\\
= &$\top$& H8\\

\end{tabular}
\end{center}
\end{proof}

\begin{proposition}
\label{prop:Aplus plus}
For any SMA (resp.~LQMA, UQMA, DPL, APL, and WSA) $\mathbb{A}$ and any HSMA (resp.~HLQMA, HUQMA, HDPL, HAPL, and HWSA) $\mathbb{H}$:
\begin{center}
$ \mathbb{A} \cong (\mathbb{A}^+)_+ \quad \mbox{and}\quad \mathbb{H} \cong (\mathbb{H}_+)^+.$
\end{center}
These correspondences restrict appropriately to the relevant classes of perfect algebras and perfect heterogeneous algebras.
\end{proposition}
\begin{proof}
It immediately follows from Proposition \ref{prop:from single to multi}, Corollary \ref{cor:sub single to multi}, Proposition \ref{prop:reverse engineering} and Corollary \ref{cor:sub multi to single}.
\end{proof}

\subsection{Canonical extensions of heterogeneous algebras}
\label{Canonical extensions}
Canonicity in the multi-type environment is used both to provide complete semantics for a large class of axiomatic extensions of the basic logic (semi De Morgan logic in the present case), and to prove the conservativity  of its associated display calculus (cf.~Section \ref{ssec: conservativity}). In the present section, we define the canonical extension $\mathbb{H}^\delta$ of any heterogeneous algebra $\mathbb{H}$ introduced  in Section \ref{Heterogeneous presentation}  by instantiating the general definition discussed in \cite{linearlogPdisplayed}. This makes it possible to define the canonical extension of any SMA $\mathbb{A}$ as a perfect SMA $(\mathbb{A}^+)^\delta_+$.
 We then show that this definition coincides with the definition given in \cite[Section 4]{Palma2005}. In what follows, we let $\mathbb{L}^\delta$ and $\mathbb{A}^\delta$ denote the canonical extensions of the distributive lattice $\mathbb{L} $ and of the De Morgan algebra (resp.~Boolean algebra) $\mathbb{D}$  respectively, and $e^\pi$ and $h^\delta$ denote the $\pi$-extensions of $e$ and $h$\footnote{The order-theoretic properties of $h$ guarantee that the $\sigma$-extension and the $\pi$-extension of $h$ coincide. This is why we use $h^\delta$ to denote the resulting extension.}, respectively. We refer to \cite{gehrke2001bounded} for the relevant definitions.

\begin{proposition}
\label{prop:canonical extensions}
If $(\mathbb{L}, \mathbb{D}, e, h)$ is an HSMA (resp.~HDPL, HLQMA, HUQMA, HAPL, and HWSA),
then $(\mathbb{L}^\delta, \mathbb{D}^\delta, e^\pi, h^\delta)$ is a perfect HSMA (resp.~perfect HDPL, perfect HLQMA, perfect HUQMA, perfect HAPL and perfect HWSA).
\end{proposition}

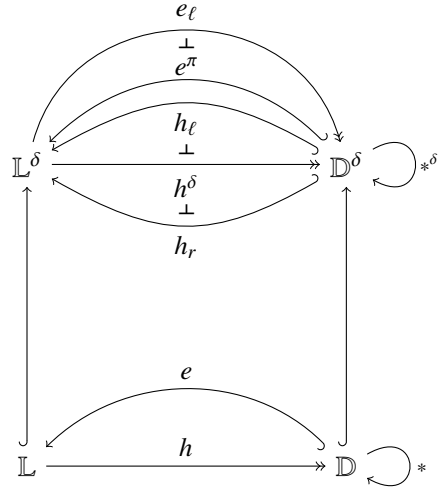
\[
\begin{tikzpicture}[node/.style={circle, draw, fill=black}, scale=1]
\node (DL) at (-2.1,-2.0) {$\mathbb{L}$};
\node (DM) at (2.1, -2.0) {$\mathbb{D}$};
\node (a) at (3.1,-2.1) {$^\ast$};
\node (a*) at (3.2,1.95) {$^{\ast^\delta}$};
\draw [->]  (DM)  to [out= 30,in= -30, looseness = 7](DM);
\draw [->>] (DL)  to node[above] {\small $h$}  (DM);
\draw [right hook->] (DM) to [out=135, in = 45, looseness=1]   node[above] {\small $e$}  (DL);

\node (DL delta) at (-2.1,2.0) {$\mathbb{L}^{\delta}$};
\node (DM delta) at (2.1,2.0) {$\mathbb{D}^{\delta}$};
\node (adje) at (0, 3.6) {\rotatebox[origin=c]{-270}{$\vdash$}};
\draw [right hook ->] (DL) to  (DL delta);
\draw [right hook ->] (DM)  to (DM delta);

\draw [->]  (DM delta)  to [out= 30, in = -30,looseness = 5.5] (DM delta);

\draw [->>] (DL delta)  to [out=75,in=105,looseness= 1.3] node[above] {\small $e_\ell$}  (DM delta);

\draw [left hook->] (DM delta) to [out=135, in= 45, looseness= 1.1] (DL delta);

\node (epi) at (0, 3.3) {\small $e^{\pi}$};

\draw [->>] (DL delta)  to node[below] {\small$h^{\delta}$} (DM delta);
\node (adje) at (0, 1.4) {\rotatebox[origin=c]{-270}{$\vdash$}};
\node (adje) at (0, 2.2) {\rotatebox[origin=c]{-270}{$\vdash$}};

\draw [left hook->] (DM delta)  to [out=-150, in = -30, looseness= 1.2] node[below] {\small $h_r$}  (DL delta);

\draw [right hook->] (DM delta)  to [out= 150, in= 30, looseness= 1.2] node[below] {\small $h_\ell$}(DL delta);

\end{tikzpicture}
\]

\begin{proof}
Firstly, $\mathbb{L}^\delta$ and $\mathbb{D}^\delta$ are a perfect distributive lattice and a perfect De Morgan algebra (resp.~perfect Boolean algebra) respectively. Secondly, since $h$ is a surjective homomorphism, $h$ is both finitely meet-preserving and finitely join-preserving.  Hence, as is well known, $h^\delta$ is surjective, and completely meet- (join-) preserving \cite[Theorem 3.7]{gehrke2004bounded}. Since $h$ is also smooth, this shows that $h^\delta = h^\pi = h^\sigma$ is a complete homomorphism.  Thirdly, since $e$ is finitely meet-preserving, $e^\pi$ is completely meet-preserving, and it immediately follows from the definition of $\pi$-extension that $e^\pi$ is an order-embedding \cite[Corollary 2.25]{gehrke2004bounded}. The identity  $e^\pi(0) = 0$ clearly holds, since $\mathbb{A}$ is a subalgebra of $\mathbb{A^\delta}$. Moreover, $Id_\mathbb{D} = h\circ e$ is canonical by \cite[Proposition 14]{gehrke2000monotone}. This is enough to show that if $(\mathbb{L}, \mathbb{D}, e, h)$ is a SMA (resp.~DPL), then $(\mathbb{L}^\delta, \mathbb{D}^\delta, e^\pi, h^\delta)$ is a perfect HSMA (resp.~perfect HDPL).

Since (H6a), (H6b), (H7) and (H8) are analytic inductive (cf.~Definition \ref{def:analytic inductive ineq}),  they are canonical. So their corresponding heterogeneous  algebras are perfect.
\end{proof}
In the environment of perfect heterogeneous algebras, completely join (resp.~meet) preserving maps have right (resp.~left) adjoints. These adjoints guarantee the soundness of all display rules  in the display calculi introduced in the next section.

In \cite{Palma2005}, C.~Palma studied the canonical extensions of semi De Morgan algebras using insights from the Sahlqvist theory of Distributive Modal Logic. She recognized that not all inequalities in the axiomatization of SMA are Sahlqvist, and circumvented this problem by introducing the following term-equivalent presentation of SMAs. 
\begin{definition}[\cite{Palma2005}, Definition 4.1.2]
For any SMA $\mathbb{A} = (L, \wedge, \vee, ', \top, \bot)$, let $\mathbb{S}_{\mathbb{A}} = (L, \wedge, \vee, \vartriangleright, \wbox, \top, \bot)$ be such that $\wbox$ and $ \vartriangleright$ are unary operations respectively defined by the assignments $a\mapsto a''$  and $a\mapsto a'$.
\end{definition}
Palma showed that the algebras corresponding to SMAs via the construction above are exactly those $\{\wbox, \vartriangleright\}$-reducts of Distributive Modal Algebras satisfying the following additional axioms: 
\begin{enumerate}
\item $\vartriangleright \top \leq \bot$;
\item $\wbox a \leq\, \vartriangleright \vartriangleright a$;
\item $\vartriangleright \vartriangleright a \leq \wbox a$;
\item ${{\wbox}{\vartriangleright}}\, a \leq\, \vartriangleright a$;
\item $\vartriangleright a \leq {{\wbox}{\vartriangleright}}\, a$.
\end{enumerate}
The axioms above can be straightforwardly verified to be Sahlqvist and hence canonical. This enables Palma to define the canonical extension $\mathbb{A}^\delta$ of $\mathbb{A}$ as the $\{\vartriangleright\}$-reduct of $\mathbb{S}^\sigma_{\mathbb{A}} = (\mathbb{L}^\sigma, \vartriangleright^\pi, \wbox^\pi)$. The following lemma immediately implies that $\mathbb{A}^\delta$ coincides with  $ (\mathbb{A}^{+\delta})_+$.

\begin{lemma}
For any SMA $\mathbb{A}$,  letting $\mathbb{S}_\mathbb{A}$ be defined as above, 
\begin{enumerate}
\item $\wbox^\pi = e^\pi \circ h^\delta$;
\item $\vartriangleright^\pi = e^\pi \circ *^\delta \circ h^\delta$.
\end{enumerate}
\end{lemma}
\begin{proof}
By the definitions of $\wbox, \vartriangleright,$ $e$ and $h$ (cf.~beginning of Section 3.1), 
\begin{center}
\begin{tabular}{rll}
$\wbox^\pi$ =& $('')^\pi$ & definition of $\wbox$ \\
=& $(e \circ h)^\pi$ & definitions of $e$ and $h$\\
=&  $e^\pi \circ h^\pi$&  \cite[Lemma 3.3, Corollary 2.25]{gehrke2004bounded}\\
=&$e^\pi \circ h^\delta$& $h$ is smooth\\
\end{tabular}
\end{center}

\begin{center}
\begin{tabular}{rll}
$\vartriangleright^\pi$ =& $(')^\pi$ & definition of $\vartriangleright$ \\
=& $(e \circ * \circ h)^\pi$ & definition of $'$\\
=&  $e^\pi \circ *^\pi \circ h^\pi$&  \cite[Lemma 3.3, Corollary 2.25]{gehrke2004bounded}\\
=&$e^\pi \circ *^\delta \circ h^\delta$& $*$ and $h$ are smooth\\
\end{tabular}
\end{center}
\end{proof}

\section{Multi-type  presentation of semi De Morgan logic and its extensions }\label{sec: multi-type language}
In Section \ref{Heterogeneous presentation} we showed that heterogeneous semi De Morgan algebras  are  equivalent presentations of  semi De Morgan algebras. This provides a semantic motivation for introducing the multi-type language $\mathcal{L}_{\textrm{MT}}$, which is naturally interpreted on heterogeneous semi De Morgan algebras. The language $\mathcal{L}_{\textrm{MT}}$ consists of terms of types $\mathsf{DL}$ and $\mathsf{K}$, defined as follows:
\begin{center}
\begin{tabular}{rll}
$\mathsf{DL}\ni  A$& ::=&$  \,p \mid \, \wbox\alpha \mid \xtop \mid \xbot \mid A \xand A \mid A \xor A $\\
$\mathsf{K}\ni \alpha$& ::=&$ \, \circ A \mid \dtop \mid \dbot \mid \rdneg\alpha \mid \alpha \dor \alpha \mid  \alpha \dand \alpha  $
\end{tabular}
\end{center}
The interpretation of $\mathcal{L}_{\textrm{MT}}$-terms  into heterogeneous  algebras  is defined as the easy generalization of the interpretation of propositional languages in universal algebra; namely, the heterogeneous operation $e$ interprets the connective $\wbox$, the heterogeneous operation $h$ interprets the connective $\circ$, and $\mathsf{DL}$-terms (resp.~$\mathsf{K}$-terms) are interpreted in the first (resp.~second) component  of heterogeneous  algebras.  

The toggle between single-type algebras and their heterogeneous counterparts (cf.\ Sections \ref{Heterogeneous presentation}) is reflected syntactically by the  translations $(\cdot)^{\tau}: \mathcal{L}\to \mathcal{L}_{\textrm{MT}}$
 defined as follows:
\begin{center}
\begin{tabular}{rcl}
$p^{\tau}$   & $::=$ & $p$\\
$\top^{\tau}$ & $::=$ & $\xtop$ \\
$\bot^{\tau}$ & $::=$ & $\xbot$ \\
$(A \wedge B)^{\tau}$ & $::= $ & $A^{\tau} \xand B^{\tau}$\\
$(A \vee B)^{\tau}$ & $::=$ & $ A^{\tau} \xor B^{\tau}$ \\
$({\neg} A)^{\tau}$ & $::=$ & ${\wbox \dneg\bbox} A^{\tau}$ \\
\end{tabular}
\end{center}

 Recall that $\mathbb{A}^+$ denotes the heterogeneous algebra associated with the single-type algebra $\mathbb{A}$ (cf.~Definition \ref{def: aplus}). The following proposition is proved by a routine induction on  $\mathcal{L}$-formulas.
\begin{proposition}
\label{prop:consequence preserved and reflected}
For all $\mathcal{L}$-formulas $A$ and $B$ and any $\mathbb{A} \in \{$SMA, LQMA, UQMA, DPL, APL, WSA$\}$,
$$\mathbb{A}\models A\fCenter B \quad \mbox{ iff }\quad \mathbb{A}^+ \models A^\tau \fCenter B^\tau.$$
\end{proposition}

We are now in a position to translate the characteristic axioms of every logic mentioned in Section \ref{ssec:Hilbert system} into $\mathcal{L}_{\textrm{MT}}$. Together with Proposition \ref{prop:Aplus plus}, the proposition above guarantees that the translation of each of the axioms below is valid on the corresponding class of heterogeneous algebras.

\begin{center}
\begin{tabular}{r l l}
$ \neg\neg A \xand  \neg\neg B \fCenter \neg\neg(A\xand B) \rightsquigarrow $ &
  ${{\wbox}{\dneg}{\circ}{\wbox}{\rdneg}{\circ}} (A \xand B)^\tau \fCenter{{\wbox}{\dneg}{\circ}{\wbox}{\rdneg}{\circ}} A^\tau \xand {{\wbox}{\dneg}{\circ}{\wbox}{\rdneg}{\circ}} B^\tau $ &$(i)$
  \\

\\
$\neg A \vdash \neg\neg\neg A \rightsquigarrow$ &${{\wbox}{\dneg}{\circ}}A^\tau \fCenter{{\wbox}{\dneg}{\circ}{\wbox}{\dneg}{\circ}{\wbox}{\dneg}{\circ}} A^\tau$ &$ (ii)$\\
\\
$\neg\neg\neg A \vdash \neg A \rightsquigarrow$ &${{\wbox}{\dneg}{\circ}{\wbox}{\dneg}{\circ}{\wbox}{\dneg}{\circ}}A^\tau \fCenter {{\wbox}{\dneg}{\circ}}A^\tau$ &$ (iii)$\\
  \\
 $ \neg A \xand  \neg B \fCenter \neg(A\xor B) \ \rightsquigarrow $ &
 ${{\wbox}{\dneg}{\circ}}(A \xor B)^\tau \fCenter {{\wbox}{\dneg}{\circ}}A^\tau \xand {{\wbox}{\dneg}{\circ}} B^\tau $ &$(iv)$\\
 
 \\
$\top \fCenter \neg\bot\rightsquigarrow $&
$   \xtop \fCenter {{\wbox}{\dneg}{\circ}}\xbot $  &$(v)$\\
\\
$
\neg\top \fCenter \bot\ \rightsquigarrow $&
$ {{\wbox}{\dneg}{\circ}}\xtop \fCenter \xbot$  &$(vi)$\\
  \\
$ A \fCenter  \neg\neg A \rightsquigarrow $ & $A^\tau \fCenter {{\wbox}{\dneg}{\circ}{\wbox}{\dneg}{\circ}}A^\tau$   &$(vii)$\\
\\
  $\neg\neg A \fCenter A \rightsquigarrow $ & ${{\wbox}{\dneg}{\circ}{\wbox}{\dneg}{\circ}}A^\tau \fCenter  A^\tau$ & $(viii)$\\
 \\ 
$\neg A \wedge \neg\neg A \fCenter \bot \rightsquigarrow$&$ {{\wbox}{\dneg}{\circ}}A^\tau \wedge  {{\wbox}{\dneg}{\circ}{\wbox}{\dneg}{\circ}}A^\tau \fCenter \xbot $ &$(ix)$\\
  
  \\
$ A \wedge \neg A \fCenter \bot \rightsquigarrow$&$A^\tau \wedge  {{\wbox}{\dneg}{\circ}} A^\tau  \fCenter \xbot $& $(x)$\\
\\
$\top \fCenter \neg A \vee \neg\neg A \rightsquigarrow$& $\xtop \fCenter {{\wbox}{\dneg}{\circ}}A^\tau \vee {{\wbox}{\dneg}{\circ}{\wbox}{\dneg}{\circ}} A^\tau$  &$(xi)$
   \end{tabular}
 \end{center}

  Notice that the defining identities of heterogeneous  algebras (cf.~Definition  \ref{def:heterogeneous algebras}) can be expressed as {\em analytic inductive} $\mathcal{L}_{\textrm{MT}}$-inequalities (cf.~Definition \ref{def:analytic inductive ineq}). Hence, these inequalities can be used to generate the analytic rules of the calculus introduced in Section \ref{sec:proper display}, with a methodology analogous to the one introduced in \cite{greco2016unified}. As we will discuss in Section \ref{ssec: completeness}, the inequalities $(i)$-$(xi)$ are derivable in the calculus obtained in this way.

\section{Proper Display Calculi for semi De Morgan logic and its extensions}\label{sec:proper display}

In the present section, we introduce  proper multi-type display calculi for semi De Morgan logic and its extensions. The language manipulated by these calculi has types $\mathsf{DL}$ and $\mathsf{K}$, and  is built up from structural and operational (aka logical) connectives. In the tables of Section \ref{ssec:language of DSM},  each structural connective corresponding to a logical connective which belongs to the  family  $\mathcal{F}$ (resp.~$\mathcal{G}$, $\mathcal{H}$) defined in Section \ref{sec:analytic inductive ineq} is denoted by decorating that logical connective with $\hat{\phantom{a}}$ (resp.~$\check{\phantom{a}}$, $\tilde{\phantom{a}}$).\footnote{\label{footnote:def precedent succedent pos}For any  sequent $x\vdash y$, we define the signed generation trees $+x$ and $-y$ by labelling the root of the generation tree of $x$ (resp.\ $y$) with the sign $+$ (resp.\ $-$), and then propagating the sign to all nodes according to the polarity of the coordinate of the connective assigned to each node. Positive (resp.\ negative) coordinates propagate the same (resp.\ opposite) sign to the corresponding child node.  Then, a substructure $z$ in $x\vdash y$ is in {\em precedent} (resp.\ {\em succedent}) {\em position} if the sign of its root node as a subtree of $+x$ or $-y$ is  $+$ (resp.\ $-$).}

\subsection{Language}
\label{ssec:language of DSM}

\paragraph{Structural and operational terms.} 
\begin{center}
\begin{tabular}{l}
$\mathsf{DL}\left\{\begin{array}{l}
A ::= \,p \mid \xtop \mid \xbot \mid \wbox \alpha  \mid A \xand A \mid A \xor A\\ 
 \\
 X ::= A \mid \XTOP \mid \XBOT \mid \WBOX\Gamma \mid \LH\Gamma \mid \RH\Gamma \mid  X \XAND X  \mid X \XOR X \mid X \XCRARR X \mid X \XRARR X \\
\end{array} \right.$
\\
\\

$\phantom{D}\mathsf{K}\left\{\begin{array}{l}
\alpha ::= \,  \dtop  \mid \dbot \mid \wdia A \mid  {\rdneg\alpha} \mid \alpha \dand \alpha \mid \alpha \dor \alpha\\ 
 \\
\Gamma ::= \alpha \mid \DTOP  \mid \DBOT  \mid \WCIR X \mid \BDIA X \mid \DNEG \Gamma \mid \Gamma \DAND \Gamma  \mid \Gamma \DOR \Gamma \mid \Gamma \DCRARR \Gamma  \mid \Gamma \DRARR \Gamma \\
\end{array} \right.$
\end{tabular}
\end{center}

Interpretation of pure-type structural connectives as their logical counterparts\footnote{In the synoptic table, the operational symbols which occur only at the structural level will appear between round brackets.}:\\

{\fns
\begin{center}
\begin{tabular}{|c|c|c|c|c|c|c|c|c|c|c|c|c|c|}
\cline{1-14}
 \mc{6}{|c|}{$ \phantom{\Big(}\mathsf{DL}\phantom{\Big )}$}& \mc{8}{c|}{ $\mathsf{K}$}                                                                                              \\
\cline{1-14}
$\phantom{\Big(}\XTOP\phantom{\Big )} $& $\XAND$ & \mc{1}{c|}{$\XCRARR$} & \mc{1}{c|}{$\XBOT$} & $\XOR$ & \mc{1}{c|}{$\XRARR$}& \mc{1}{c|}{$\DTOP$} & $\DAND$ & \mc{1}{c|}{$\DCRARR$} & \mc{1}{c|}{$\DBOT$} & $\DOR$ & \mc{1}{c|}{$\DRARR$} & \mc{2}{c|}{$\DNEG$} \\
\cline{1-14}
$\phantom{\Big )}\,\xtop\,\phantom{\Big )}$    & $\,\xand\,$   & \mc{1}{c|}{$(\xcrarr)$}       & \mc{1}{c|}{$\,\xbot\,$}    & $\,\xor\,$    & \mc{1}{c|}{$(\xrarr)$}& $\,\dtop\,$    & $\,\dand\,$   & \mc{1}{c|}{$(\dcrarr)$}       & \mc{1}{c|}{$\,\dbot\,$}    & $\,\dor\,$    & \mc{1}{c|}{$(\drarr)$}   &\mc{1}{c|}{$\,\dneg\,$}&\mc{1}{c|}{$\,\rdneg\,$}\\
\cline{1-14}
\end{tabular}
\end{center}
}
 Interpretation of heterogeneous  structural connectives as their logical counterparts:
{\fns
\begin{center}
\begin{tabular}{|c|c|clclc|c|}
\cline{1-6}
 \mc{2}{|c|}{$\phantom{\Big )}\mathsf{DL} \to \mathsf{K}\phantom{\Big )}$}& \mc{1}{c|}{$\mathsf{K} \to \mathsf{DL}$} & \mc{1}{c|}{$\mathsf{K} \to \mathsf{DL}$} & \mc{1}{c|}{$\mathsf{K} \to \mathsf{DL}$} & \mc{1}{c|}{$\mathsf{DL} \to \mathsf{K}$}                                                                                              \\
\cline{1-6}
 \mc{2}{|c|}{$\WCIR$} & \mc{1}{c|}{$\phantom{\Big )}\LH\phantom{\Big )}$}& \mc{1}{c|}{$\RH$} & \mc{1}{c|}{$\WBOX$}& \mc{1}{c|}{$\phantom{\Big(}\BDIA\phantom{\Big)}$} \\
\cline{1-6}
$\phantom{\Big)}\wdia\phantom{\Big )}$&$\bbox$ &\mc{1}{c|}{$(\lh)$}& \mc{1}{c|}{$(\rh)$}& \mc{1}{c|}{$\,\wbox\,$}& \mc{1}{c|}{$(\bdia)$}  \\
\cline{1-6}
\end{tabular}
\end{center}
}

Algebraic interpretation of heterogeneous  structural connectives as operations in perfect HSM-algebras (see Lemma \ref{prop:canonical extensions}).\\

{\fns
\begin{center}
\begin{tabular}{|c|c|clclc|c|}
\cline{1-6}
 \mc{2}{|c|}{$\phantom{\Big(} \mathsf{DL} \to \mathsf{K}$}& \mc{2}{c|}{$\mathsf{K} \to \mathsf{DL}$} & \mc{1}{c|}{$\mathsf{K} \to \mathsf{DL}$} & \mc{1}{c|}{$\mathsf{DL} \to \mathsf{K}$}                                                                                              \\
\cline{1-6}
\mc{2}{|c|}{$\phantom{\Big(}\ \,\WCIR\ \,$}&\mc{1}{c|}{$\ \,\LH\ \,$}& \mc{1}{c|}{$\,\RH\,$}& \mc{1}{c|}{$\,\WBOX\,$}& \mc{1}{c|}{$\,\BDIA\,$}  \\
\cline{1-6}
 \mc{2}{|c|}{$h$} &\mc{1}{c|}{$\ \,h_\ell\ \,$}& \mc{1}{c|}{$\phantom{\Big(}\,h_r\,$}& \mc{1}{c|}{$\,e\,$}& \mc{1}{c|}{$\,e_\ell\,$}  \\
\cline{1-6}
\end{tabular}
\end{center}
}

\subsection{Multi-type display calculi for semi De Morgan logic and its extensions}
\label{ssec:display calculi}
In what follows, structures of type $\mathsf{DL}$ are denoted by the variables $X, Y, Z$, and $W$; structures of type $\mathsf{A}$ are denoted by the variables $\Gamma, \Delta, \Theta$ and $\Pi$.

\begin{enumerate}
\item[-] The proper display calculus for semi De Morgan logic D.SM consists of the following rules:
\begin{itemize}
\item Identity and cut rules
\begin{center}
\begin{tabular}{rcl}

\AxiomC{\phantom{$X \fCenter A$}}
\LeftLabel{\scriptsize Id}
\UI $p \fCenter p$
\DisplayProof
 &
\AX $X \fCenter A$
\AX $A \fCenter Y$
\RightLabel{\scriptsize $\mathrm{Cut}_\mathsf{L}$}
\BI$X \fCenter Y$
\DisplayProof
&
\AX $\Gamma \fCenter \alpha$
\AX $\alpha \fCenter \Delta$
\RightLabel{\scriptsize $\mathrm{Cut}_\mathsf{D}$}
\BI$\Gamma \fCenter \Delta$
\DisplayProof
 \\
\end{tabular}
\end{center}

\item Pure $\mathsf{DL}$-type display rules
\begin{center}
\begin{tabular}{rcl}
\AX $X \XAND Y \fCenter Z$
\LeftLabel{\scriptsize $\mathrm{res}_\mathsf{L}$}
\doubleLine
\UI$Y \fCenter X \XRARR Z$
\DisplayProof
&
$\quad$
&
\AX $X \fCenter Y \XOR Z $
\RightLabel{\scriptsize $\mathrm{res}_\mathsf{L}$}
\doubleLine
\UI$Y \XCRARR X \fCenter Z$
\DisplayProof
\end{tabular}
\end{center}

\item Pure $\mathsf{K}$-type display rules

\begin{center}
\begin{tabular}{rcl}
\AX $\Gamma \DAND \Delta \fCenter \Theta$
\LeftLabel{\scriptsize $\mathrm{res}_\mathsf{D}$}
\doubleLine
\UI$\Delta\fCenter \Gamma \DRARR \Theta$
\DisplayProof
&
&
\AX $\Gamma\fCenter \Delta \DOR \Theta$
\RightLabel{\scriptsize $\mathrm{res}_\mathsf{D}$}
\doubleLine
\UI$\Delta \DCRARR \Gamma \fCenter \Theta$
\DisplayProof
\\
\\

\AX $\DNEG \Gamma \fCenter \Delta$
\LeftLabel{\scriptsize $\mathrm{adj}_*$}
\doubleLine
\UI$\DNEG \Delta \fCenter \Gamma$
\DisplayProof
&
&
\AX $\Gamma \fCenter \DNEG \Delta$
\RightLabel{\scriptsize $\mathrm{adj}_*$}
\doubleLine
\UI  $ \Delta \fCenter \DNEG\Gamma$
\DisplayProof
\end{tabular}
\end{center}

\item Multi-type display rules
\begin{center}
\begin{tabular}{rcl}
\AX $X  \fCenter \WBOX \Gamma$
\LeftLabel{\scriptsize $\mathrm{adj}_\mathsf{LD}$}
\doubleLine
\UI$ \BDIA  X \fCenter \Gamma$
\DisplayProof
&
\AX $\WCIR X \fCenter \Gamma $
\LL{\scriptsize $\mathrm{adj}_\mathsf{DL}$}
\doubleLine
\UI$X \fCenter \RH \Gamma$
\DisplayProof
&
\AX $\Gamma \fCenter \WCIR X$
\RL{\scriptsize  $\mathrm{adj}_\mathsf{DL}$}
\doubleLine
\UI$\LH \Gamma \fCenter X$
\DisplayProof
\end{tabular}
\end{center}

\item Pure $\mathsf{DL}$-type structural rules

\begin{center}
\begin{tabular}{rcl}
\AX $X \fCenter Y$
\LeftLabel{\scriptsize $\XTOP$}
\doubleLine
\UI $X \XAND \XTOP \fCenter Y$
\DisplayProof
 &
 &
\AX$X \fCenter Y$
\RightLabel{\scriptsize $\XBOT$}
\doubleLine
\UI$X \fCenter Y \XOR \XBOT$
\DisplayProof
\\
\\
\AX$X \XAND Y \fCenter Z$
\LeftLabel{\scriptsize $\mathrm{E}_\mathsf{L}$}
\UI$Y \XAND X \fCenter Z$
\DisplayProof
 &
 &
\AX$X \fCenter Y \XOR Z$
\RightLabel{\scriptsize $\mathrm{E}_\mathsf{L}$}
\UI $X \fCenter Z \XOR Y$
\DisplayProof
\\
\\
\AX $(X \XAND Y)  \XAND Z \fCenter W$
\LeftLabel{\scriptsize $\mathrm{A}_\mathsf{L}$}
\doubleLine
\UI$X \XAND (Y  \XAND Z) \fCenter Z$
\DisplayProof
 &&
\AX$X \fCenter (Y \XOR Z) \XOR W$
\RightLabel{\scriptsize $\mathrm{A}_\mathsf{L}$}
\doubleLine
\UI $X \fCenter Y \XOR (Z \XOR W)$
\DisplayProof
\\
\\
\AX$X \fCenter Y$
\LeftLabel{\scriptsize $\mathrm{W}_\mathsf{L}$}
\UI$X \XAND Z \fCenter Y$
\DisplayProof
 &&
\AX $X \fCenter Y$
\RightLabel{\scriptsize $\mathrm{W}_\mathsf{L}$}
\UI $X \fCenter Y \XOR Z$
\DisplayProof
\\
\\
\AX$X \XAND X \fCenter Y$
\LeftLabel{\scriptsize $\mathrm{C}_\mathsf{L}$}
\UI $X \fCenter Y$
\DisplayProof
 &&
\AX$X \fCenter Y \XOR Y$
\RightLabel{\scriptsize $\mathrm{C}_\mathsf{L}$}
\UI $X \fCenter Y $
\DisplayProof
\\

\end{tabular}
\end{center}

\item Pure $\mathsf{K}$-type structural rules
\begin{center}
\begin{tabular}{rcl}
\AX $\Gamma \fCenter \Delta$
\LeftLabel{\scriptsize \DTOP}
\doubleLine
\UI $\Gamma \DAND \DTOP \fCenter \Delta$
\DisplayProof
 &&
\AX$\Gamma \fCenter \Delta$
\RightLabel{\scriptsize \DBOT}
\doubleLine
\UI$\Gamma \fCenter \Delta \DOR \DBOT$
\DisplayProof
\\
\\
\AX$\Gamma \DAND \Delta \fCenter \Theta$
\LeftLabel{\scriptsize $\mathrm{E}_\mathsf{D}$}
\UI$\Delta \DAND \Gamma \fCenter \Theta$
\DisplayProof
 &&
\AX$\Gamma \fCenter \Delta \DOR \Theta$
\RightLabel{\scriptsize $\mathrm{E}_\mathsf{D}$}
\UI $\Gamma \fCenter \Theta \DOR \Delta$
\DisplayProof
\\
\\
\AX $(\Gamma \DAND \Delta)  \DAND \Theta \fCenter \Pi$
\LeftLabel{\scriptsize $\mathrm{A}_\mathsf{D}$}
\doubleLine
\UI$\Gamma \DAND (\Delta  \DAND \Theta) \fCenter \Theta$
\DisplayProof
 &&
\AX$\Gamma \fCenter (\Delta \DOR \Theta) \DOR \Pi$
\RightLabel{\scriptsize $\mathrm{A}_\mathsf{D}$}
\doubleLine
\UI $\Gamma \fCenter \Delta \DOR (\Theta \DOR \Pi)$
\DisplayProof
\\
\\
\AX$\Gamma \fCenter \Delta$
\LeftLabel{\scriptsize $\mathrm{W}_\mathsf{D}$}
\UI$\Gamma \DAND \Theta \fCenter \Delta$
\DisplayProof
 &&
\AX $\Gamma \fCenter \Delta$
\RightLabel{\scriptsize $\mathrm{W}_\mathsf{D}$}
\UI $\Gamma \fCenter \Delta \DOR \Theta$
\DisplayProof
\\
\\
\AX$\Gamma \DAND \Gamma \fCenter \Delta$
\LeftLabel{\scriptsize $\mathrm{C}_\mathsf{D}$}
\UI $\Gamma \fCenter \Delta$
\DisplayProof
 &&
\AX$\Gamma \fCenter \Delta \DOR \Delta$
\RightLabel{\scriptsize $\mathrm{C}_\mathsf{D}$}
\UI $\Gamma \fCenter \Delta $
\DisplayProof
\\
\\
 \mc{3}{c}{\AX$\Gamma \fCenter  \Delta$
\RightLabel{\scriptsize cont}
\doubleLine
\UI$\DNEG \Delta \fCenter \DNEG \Gamma$
\DisplayProof}
\end{tabular}
\end{center}
\item Multi-type structural rules
\begin{center}
\begin{tabular}{rcl}
\AX $X  \fCenter Y$
\LeftLabel{\scriptsize \WCIR}
\UI$\WCIR X \fCenter \WCIR Y$
\DisplayProof
&
&
\AX $\LH \Gamma  \fCenter \RH \Delta$
\RL{\scriptsize \BCIR}
\UI$\Gamma \fCenter \Delta$
\DisplayProof
\\
\\
\AX $\DTOP \fCenter \Gamma$
\LeftLabel{\scriptsize {\BDIA}\DTOP}
\UI$\BDIA\XTOP \fCenter \Gamma$
\DisplayProof
&
&
\AX $X  \fCenter \WBOX\DBOT$
\RL{\scriptsize {\WBOX}\DBOT}
\UI$X \fCenter \XBOT$
\DisplayProof
\\
\\
\mc{3}{c}{
\AX $\Gamma  \fCenter \WCIR\WBOX\Delta$
\RL{\scriptsize {\WCIR}\WBOX}
\doubleLine
\UI $\Gamma \fCenter \Delta$
\DisplayProof}
\end{tabular}
\end{center}

\item Pure $\mathsf{DL}$-type operational rules

\begin{center}
\begin{tabular}{rcl}
\AX$\XTOP \fCenter X$
\LeftLabel{\scriptsize \xtop}
\UI$\xtop \fCenter X$
\DisplayProof
 &&
\AxiomC{\phantom{$\XTOP \fCenter X$}}
\RightLabel{\scriptsize \xtop}
\UI $\XTOP \fCenter \xtop$
\DisplayProof
\\
\\
\AxiomC{\phantom{$X \fCenter \XBOT$}}
\LeftLabel{\scriptsize \xbot}
\UI $\xbot \fCenter \XBOT$
\DisplayProof
 &&
\AX$X \fCenter \XBOT$
\RightLabel{\scriptsize \xbot}
\UI$X \fCenter \xbot$
\DisplayProof
\\
\\
\AX$A \XAND B \fCenter X$
\LeftLabel{\scriptsize \xand}
\UI$A \xand B \fCenter X$
\DisplayProof
 &&
\AX$X \fCenter A$
\AX$Y \fCenter B$
\RightLabel{\scriptsize \xand}
\BI$X \XAND Y \fCenter A \xand B$
\DisplayProof
\\
\\
\AX$A \fCenter X$
\AX$B \fCenter Y$
\LeftLabel{\scriptsize \xor}
\BI$A \xor B \fCenter X {\XOR} Y$
\DisplayProof
 &&
\AX$X \fCenter A \XOR B$
\RightLabel{\scriptsize \xor}
\UI$X \fCenter A \xor B $
\DisplayProof
\end{tabular}
\end{center}

\item Pure $\mathsf{K}$-type operational rules

\begin{center}
\begin{tabular}{rcl}
\AX$\DTOP \fCenter \Gamma$
\LeftLabel{\scriptsize \dtop}
\UI$\dtop \fCenter \Gamma$
\DisplayProof
 &&
\AxiomC{\phantom{$\DTOP \fCenter \Gamma$}}
\RightLabel{\scriptsize \dtop}
\UI $\DTOP \fCenter \dtop$
\DisplayProof
 \\
\\
\AxiomC{\phantom{$\Gamma \fCenter \DBOT$}}
\LeftLabel{\scriptsize \dbot}
\UI $\dbot \fCenter \DBOT$
\DisplayProof
 &&
\AX$\Gamma \fCenter \DBOT$
\RightLabel{\scriptsize \dbot}
\UI$\Gamma \fCenter \dbot$
\DisplayProof
\\
\\
\AX$\alpha \DAND \beta \fCenter \Gamma$
\LeftLabel{\scriptsize \dand}
\UI$\alpha \dand \beta \fCenter \Gamma$
\DisplayProof
 &&
\AX$\Gamma \fCenter \alpha$
\AX$\Delta \fCenter \beta$
\RightLabel{\scriptsize \dand}
\BI$\Gamma \DAND \Delta \fCenter \alpha \dand \beta$
\DisplayProof
 \\
\\
\AX$\alpha \fCenter \Gamma$
\AX$\beta \fCenter \Delta$
\LeftLabel{\scriptsize \dor}
\BI$\alpha \dor \beta \fCenter \Gamma {\DOR} \Delta$
\DisplayProof
 &&
\AX$\Gamma \fCenter \alpha \DOR \beta$
\RightLabel{\scriptsize \dor}
\UI$\Gamma \fCenter \alpha \dor \beta $
\DisplayProof
\\
\\
\AX $\DNEG \alpha \fCenter \Gamma$
\LeftLabel{\scriptsize  \dneg}
\UI$\dneg \alpha \fCenter \Gamma$
\DisplayProof
 &&
\AX$\Gamma \fCenter \DNEG\alpha $
\RightLabel{\scriptsize \dneg}
\UI$\Gamma \fCenter \dneg \alpha $
\DisplayProof

\end{tabular}
\end{center}

\item Multi-type operational rules
\begin{center}
\begin{tabular}{rcl}
\AX$\WCIR A \fCenter \Gamma$
\LeftLabel{\scriptsize \wdia}
\UI$\wdia A \fCenter \Gamma$
\DisplayProof
&
&
\AX$X \fCenter \WCIR A$
\RightLabel{\scriptsize \wdia}
\UI$X \fCenter \wdia A$
\DisplayProof
\\

\\
\AX$\alpha \fCenter \Gamma$
\LeftLabel{\scriptsize \wbox}
\UI$\wbox\alpha \fCenter \WBOX\Gamma$
\DisplayProof
&
&
\AX$X \fCenter \WBOX\alpha$
\RightLabel{\scriptsize \wbox}
\UI$X \fCenter \wbox\alpha$
\DisplayProof
\end{tabular}
\end{center}
\end{itemize}

\item[-] The proper display calculus D.LQM for lower quasi De Morgan logic consists of all axiom and rules in D.SM plus the following rule:
\[
\AX$X \fCenter Y$
\RightLabel{\scriptsize LQM}
\UI$X \fCenter {\WBOX\WCIR} Y$
\DisplayProof
\]

\item[-] The proper display calculus  D.UQM for upper quasi De Morgan logic consists of all axiom and rules in D.SM plus the following rule:
\[
\AX$\LH\BDIA X \fCenter Y$
\RightLabel{\scriptsize UQM}
\UI$X \fCenter Y$
\DisplayProof
\]

\item[-] The proper display calculus D.DP  for demi pseudocomplemented lattice logic consists of all axiom and rules in D.SM  plus the following rule:
\begin{center}
\begin{tabular}{c}
\AX $\Gamma \DAND \Delta \fCenter \Sigma$
\LeftLabel{\scriptsize $\mathrm{res}_\mathsf{B}$}
\doubleLine
\UI$\Delta \fCenter \DNEG\Gamma \DOR \Sigma$
\DisplayProof
\end{tabular}
\end{center}

\item[-]The proper display calculus D.AP for almost pseudocomplemented lattice logic  consists of all axiom and rules in D.DP  plus the following rule:
\[
\AX $X \fCenter \WBOX\DNEG\WCIR Y$
\RightLabel{\scriptsize AP}
\UI$X \XAND Y\fCenter  \XBOT$
\DisplayProof
\]

\item[-]The proper display calculus D.WS for weak stone logic  consists of all axiom and rules in D.DP  plus the following rule:
\[
\AX $\BDIA X \fCenter \Delta$
\RightLabel{\scriptsize WS}
\UI$\BDIA(\WBOX{\DNEG}\WCIR X \XCRARR\XTOP) \fCenter \Delta$
\DisplayProof
\]
\end{enumerate}

\section{Properties}
\label{sec:properties}

\label{ssec: properties}
\subsection{Soundness}
\label{ssec:soundness}

In the present subsection, we outline the  verification of the soundness of the rules of D.SM (resp.~D.LQM, D.UQM D.DP, D.AP and D.WS) w.r.t.~the semantics of {\em perfect} HSMAs (resp.~HQMAs, HDPLs, HAPLs and HWSAs, see Definition \ref{def:heterogeneous algebras}). The first step consists in interpreting structural symbols as logical symbols according to their (precedent or succedent) position, as indicated at the beginning of Section \ref{sec:proper display}. This makes it possible to interpret sequents as inequalities, and rules as quasi-inequalities. For example, the rules on the left-hand side below are interpreted as the quasi-inequalities on the right-hand side:
 \begin{center}
\begin{tabular}{rcl}
\AX $X \fCenter Y$
\UI$\WCIR X \fCenter \WCIR Y$
\DisplayProof
&$\quad\rightsquigarrow\quad$&
$\forall a \forall b[a \leq  b \Rightarrow h(a) \leq h(b)]$\\
\\
\AX $\BDIA X \fCenter \Delta$
\UI$\BDIA(\WBOX{\DNEG}\WCIR X \XCRARR\XTOP) \fCenter \Delta$
\DisplayProof
&$\quad\rightsquigarrow\quad$&
$\forall a[e_\ell [e(h(a)^*)\xcrarr\top] \leq e_\ell(a)]$
\end{tabular}
\end{center}

The proof of the soundness of the rules in these display calculi then consists in verifying the validity of their corresponding quasi-inequalities in the corresponding class of perfect heterogeneous  algebras. The verification of the soundness of pure-type rules and of the introduction rules following this procedure is routine, and is omitted. The validity of the quasi-inequalities corresponding to multi-type structural rules follows straightforwardly from the observation that the quasi-inequality corresponding to each rule is obtained by running the algorithm ALBA (cf.~Section 3.4 \cite{greco2016unified}) on some of the  defining inequalities of its corresponding heterogeneous algebras.\footnote{Indeed, as discussed in \cite{greco2016unified}, the soundness of the rewriting rules of ALBA only depends on the order-theoretic properties of the interpretation of the logical connectives and their adjoints and residuals. The fact that some of these maps are not internal operations but have different domains and codomains does not make any substantial difference.} For instance, the soundness of the characteristic rule of D.WS on HWSAs follows from the validity of the inequality $(xi)$ in every HWSA (discussed in Section \ref{sec: multi-type language}) and from the soundness of the following ALBA reduction in every HWSA:
\begin{center}
\begin{tabular}{rll}
&$\forall a[\top \leq e(h(a)^*) \vee e((he(h(a)^*))^*)]$&\\
iff&$\forall a\forall b \forall c[b \leq a \,\&\, c \leq e(h(a)^*) \Rightarrow \top \leq e(h(b)^*) \vee e(h(c)^*)]$&\\
iff&$\forall a\forall b \forall c[b \leq a \,\&\, a \leq h_r(e_\ell(c)^*) \Rightarrow \top \leq e(h(b)^*) \vee e(h(c)^*)]$&\\
iff&$\forall b \forall c[b \leq h_r(e_\ell(c)^*) \Rightarrow \top \leq e(h(b)^*) \vee e(h(c)^*)]$&\\
iff&$\forall b \forall c[b \leq h_r(e_\ell(c)^*) \Rightarrow e(h(c)^*)\xcrarr\top  \leq e(h(b)^*)]$&\\
iff&$\forall b \forall c[b \leq h_r(e_\ell(c)^*) \Rightarrow b \leq h_r(e_\ell [e(h(c)^*)\xcrarr\top]^*)]$&\\
iff&$\forall c[h_r(e_\ell(c)^*) \leq h_r(e_\ell [e(h(c)^*)\xcrarr\top]^*)]$&\\
iff&$\forall c[e_\ell(c)^* \leq e_\ell [e(h(c)^*)\xcrarr\top]^*]$&$h_r$ is injective\\
iff&$\forall c[e_\ell [e(h(c)^*)\xcrarr\top] \leq e_\ell(c)]$&$^*$ is injective\\

\end{tabular}
\end{center}

\subsection{Completeness}\label{ssec: completeness}
In the present subsection, we show that the translations of the axioms and rules of $\msf{SM}$, $\msf{LQM}$,  $\msf{UQM}$, $\msf{DP}$, $\msf{AP}$ and $\msf{WS}$  are derivable in D.SM, D.LQM, D.UQM, D.DP, D.AP and D.WS, respectively. Then, the completeness of these display calculi w.r.t. the classes of SMAs, LQMAs, UQMAs, DPLs, APLs and WSAs immediately follows from  the completeness of $\msf{SM}$, $\msf{LQM}$, $\msf{UQM}$, $\msf{DP}$, $\msf{AP}$ and $\msf{WS}$ (cf.~Theorem \ref{completeness: H.SDM}).

\begin{proposition}
For every $A \in \mathcal{L}$, the sequent $A^{\tau} \fCenter A^{\tau}$ is derivable in all display calculi introduced in Section \ref{ssec:display calculi}.
\end{proposition}
\begin{proof}
By inducution on $A \in \mathcal{L}$.
The proof of base cases: $ A := \top$, $A := \bot$ and $A := p$, are straightforward and are omitted.

Inductive cases: 
\begin{itemize}
\item as to $ A := \neg B$,
{\fns
\begin{center}
\AxiomC{}
\LeftLabel{\scriptsize ind.hyp.}
\UI $B^\tau \fCenter B^\tau$
\LeftLabel{\scriptsize \WCIR}
\UI $\WCIR B^\tau \fCenter \WCIR B^\tau$
\UI $\wdia B^\tau \fCenter \WCIR B^\tau$
\UI $\wdia B^\tau \fCenter \bbox B^\tau$
\RL{\scriptsize cont}
\UI $\DNEG \bbox B^\tau \fCenter \DNEG \wdia B^\tau$
\UI $\DNEG \bbox B^\tau \fCenter \rdneg \wdia B^\tau$
\UI $\dneg \bbox B^\tau \fCenter \rdneg \wdia B^\tau$
\UI $\wbox\dneg \bbox B^\tau \fCenter \WBOX\rdneg \wdia B^\tau$
\UI $\wbox\dneg \bbox B^\tau \fCenter \wbox\rdneg \wdia B^\tau$
\DisplayProof
\end{center}
}
\item as to $ A := B \vee C$,\\
{\fns
\[\AxiomC{}
\LeftLabel{\scriptsize ind.hyp.}
\UIC{$B^\tau \fCenter B^\tau$}
\RL{\scriptsize W}
\UIC{$B^\tau \fCenter B^\tau \XOR C^\tau$}
\AxiomC{}
\LeftLabel{\scriptsize ind.hyp.}
\UIC{$C^\tau \fCenter C^\tau$}
\RL{\scriptsize W}
\UIC{$C^\tau \fCenter C^\tau \XOR B^\tau$}
\RL{\scriptsize E}
\UIC{$C^\tau \fCenter B^\tau \XOR C^\tau$}
\BIC{$B^\tau \xor C^\tau \fCenter (B^\tau \XOR C^\tau) \XOR (B^\tau \XOR C^\tau)$}
\RL{\scriptsize C}
\UIC{$B^\tau \xor C^\tau \fCenter B^\tau \xor C^\tau$}
\DisplayProof
\]
}

\item  as to $ A := B \wedge C$, \\
{\fns
\[
\AxiomC{}
\LeftLabel{\scriptsize ind.hyp.}
\UI  $B^\tau \fCenter B^\tau$
\LL{\scriptsize W}
\UI $B^\tau \XAND C^\tau \fCenter B^\tau$
\AxiomC{}
\LeftLabel{\scriptsize ind.hyp.}
\UI $C^\tau \fCenter C^\tau$
\LL{\scriptsize W}
\UI $C^\tau \XAND B^\tau \fCenter C^\tau$
\LL{\scriptsize E}
\UI $B^\tau \XAND C^\tau \fCenter C^\tau$
\BI $(B^\tau \XAND C^\tau) \XAND (B^\tau \XAND C^\tau) \fCenter B^\tau \xand C^\tau$
\LL{\scriptsize C}
\UI $B^\tau \xand C^\tau \fCenter B^\tau \xand C^\tau$
\DisplayProof
\]
}
\end{itemize}
\end{proof}

\begin{proposition}
For every $A, B \in \mathcal{L}$, if $A \vdash B$ is derivable in any logic introduced in \ref{ssec:Hilbert system}, then $A^{\tau} \fCenter B^{\tau}$ is derivable in its respective display calculus.
\end{proposition}
\begin{proof}
It is enough to show the statement of the proposition on the axioms. For the sake of readability, in what follows, we suppress the translation symbol $(\cdot)^\tau$. As to the axioms in $\msf{SM}$:
\begin{itemize}
\item $\neg\top \fCenter \bot  \quad \rightsquigarrow  \quad  \wbox\dneg\bbox\xtop \fCenter \xbot$,

{\fns
\[
\AX $\XTOP \fCenter \xtop$
\LL {\scriptsize $\mathrm{W}_\mathsf{L}$}
\UI $\LH\DNEG\DBOT \XAND \XTOP \fCenter \xtop$
\RL {\scriptsize $\XTOP$}
\UI $\LH\DNEG\DBOT\fCenter \xtop$
\UI  $\DNEG\DBOT \fCenter \WCIR\xtop$
\UI  $\DNEG\DBOT \fCenter \bbox\xtop$
\UI  $\DNEG\bbox\xtop \fCenter \DBOT$
\UI  $\dneg\bbox\xtop \fCenter \DBOT$
\UI  $\wbox\dneg\bbox\xtop \fCenter \WBOX\DBOT$
\RightLabel{\scriptsize ${\WBOX}\DBOT$}
\UI  $\wbox\dneg\bbox\xtop \fCenter \XBOT$
\UI  $\wbox\dneg\bbox\xtop \fCenter \xbot$
\DisplayProof
\]
}
\item $\top \fCenter \neg\bot \rightsquigarrow  \quad  \xtop \fCenter \wbox\rdneg\wdia\xtop$,

{\fns
\[
\AX $\xbot \fCenter \XBOT$
\RL {\scriptsize $\mathrm{W}_\mathsf{L}$}
\UI $\xbot \fCenter \XBOT \XOR \RH\DNEG\DTOP$
\RL {\scriptsize \XBOT}
\UI $\xbot \fCenter \RH\DNEG\DTOP$
\UI $\WCIR\xbot \fCenter \DNEG\DTOP$
\UI $\wdia\xbot \fCenter \DNEG\DTOP$
\UI  $\DTOP \fCenter \DNEG\wdia\xbot$
\UI  $\DTOP \fCenter \rdneg\wdia\xbot$
\LeftLabel{\scriptsize $\BDIA\XTOP$}
\UI  $\BDIA\XTOP \fCenter \rdneg\wdia\xbot$
\UI  $\XTOP \fCenter \WBOX\rdneg\wdia\xbot$
\UI  $\XTOP \fCenter \wbox\rdneg\wdia\xbot$
\UI  $\xtop \fCenter \wbox\rdneg\wdia\xbot$
\DisplayProof
\]
}

\item $\neg A \fCenter \neg\neg\neg A \quad \rightsquigarrow  \quad \wbox\dneg\bbox A \fCenter  \wbox\rdneg\wdia\wbox\dneg\bbox \wbox\rdneg\wdia A$ 

and $\neg\neg\neg A \fCenter \neg A   \quad \rightsquigarrow  \quad   \wbox\dneg\bbox\wbox\rdneg\wdia\wbox\dneg\bbox A \fCenter \wbox\rdneg\wdia A$,

{\fns
\[
\AX $A \fCenter A$
\RightLabel{\scriptsize \WCIR}
\UI  $\WCIR A \fCenter \WCIR A$
\UI  $\wdia A \fCenter \WCIR A$
\UI  $\wdia A \fCenter \bbox A$
\RightLabel{\scriptsize cont}
\UI  $\DNEG \bbox A \fCenter \DNEG\wdia A$
\UI  $\DNEG \bbox A  \fCenter \rdneg\wdia A$
\RightLabel{\scriptsize {\WCIR}\WBOX}
\UI  $\DNEG \bbox A \fCenter \WCIR\WBOX\rdneg\wdia A$
\UI  $\LH\DNEG \bbox A \fCenter \WBOX\rdneg\wdia A$
\UI  $\LH\DNEG \bbox A \fCenter \wbox\rdneg\wdia A$
\UI  $\DNEG \bbox A \fCenter \WCIR\wbox\rdneg\wdia A$
\UI  $\DNEG \bbox A  \fCenter \bbox\wbox\rdneg\wdia A$
\UI  $\DNEG\bbox\wbox\rdneg\wdia A \fCenter  \bbox A$
\UI  $\dneg\bbox\wbox\rdneg\wdia A \fCenter  \bbox A$
\UI  $\wbox\dneg\bbox\wbox\rdneg\wdia A \fCenter \WBOX \bbox A$
\RightLabel{\scriptsize \WCIR}
\UI  $\WCIR\wbox\dneg\bbox\wbox\rdneg\wdia A \fCenter \WCIR\WBOX \bbox A$
\UI  $\wdia\wbox\dneg\bbox\wbox\rdneg\wdia A \fCenter \WCIR\WBOX \bbox A$
\RightLabel{\scriptsize {\WCIR}\WBOX}
\UI  $\wdia\wbox\dneg\bbox\wbox\rdneg\wdia A \fCenter  \bbox A$
\RightLabel{\scriptsize cont}
\UI  $\DNEG \bbox A  \fCenter \DNEG\wdia\wbox\dneg\bbox\wbox\rdneg\wdia A$
\UI  $\DNEG \bbox A \fCenter \rdneg\wdia\wbox\dneg\bbox\wbox\rdneg\wdia A $
\UI  $\dneg \bbox A \fCenter \rdneg\wdia\wbox\dneg\bbox\wbox\rdneg\wdia A $
\UI  $\wbox\dneg\bbox A \fCenter \WBOX\rdneg\wdia\wbox\dneg\bbox\wbox\rdneg\wdia A$
\UI  $\wbox\dneg \bbox A \fCenter \wbox\rdneg\wdia\wbox\dneg\bbox\wbox\rdneg\wdia A$
\DisplayProof
\quad\quad
\AX $A \fCenter A$
\RightLabel{\scriptsize \WCIR}
\UI  $\WCIR A \fCenter \WCIR A$
\UI  $\wdia A \fCenter \WCIR A$
\UI  $\wdia A \fCenter \bbox A$
\RightLabel{\scriptsize cont}
\UI  $\DNEG \bbox A \fCenter \DNEG\wdia A$
\UI  $\dneg \bbox A \fCenter \DNEG\wdia A$
\UI  $\wbox\dneg \bbox A \fCenter \WBOX{\DNEG}\wdia A$
\RightLabel{\scriptsize \WCIR}
\UI  $\WCIR\wbox\dneg \bbox A \fCenter \WCIR\WBOX{\DNEG}\wdia A$
\UI  $\wdia\wbox\dneg \bbox A \fCenter \WCIR\WBOX{\DNEG}\wdia A$
\RightLabel{\scriptsize {\WCIR}\WBOX}
\UI  $\wdia\wbox\dneg \bbox A \fCenter \DNEG\wdia A$
\UI  $\wdia A \fCenter \DNEG\wdia\wbox\dneg \bbox A$
\UI  $\wdia A \fCenter \rdneg\wdia\wbox\dneg \bbox A$
\RightLabel{\scriptsize {\WCIR}\WBOX}
\UI  $\wdia A \fCenter  \WCIR\WBOX\rdneg\wdia\wbox\dneg \bbox A$
\UI  $\LH\wdia A \fCenter  \WBOX\rdneg\wdia\wbox\dneg \bbox A$
\UI  $\LH\wdia A \fCenter  \wbox\rdneg\wdia\wbox\dneg \bbox A$
\UI  $\wdia A \fCenter  \WCIR\wbox\rdneg\wdia\wbox\dneg \bbox A$
\UI  $\wdia A \fCenter  \bbox\wbox\rdneg\wdia\wbox\dneg \bbox A$
\RightLabel{\scriptsize cont}
\UI  $\DNEG\bbox\wbox\rdneg\wdia\wbox\dneg \bbox A \fCenter \DNEG\wdia A $
\UI  $\dneg\bbox\wbox\rdneg\wdia\wbox\dneg \bbox A \fCenter \DNEG\wdia A $
\UI   $\dneg\bbox\wbox\rdneg\wdia\wbox\dneg \bbox A \fCenter \rdneg\wdia A $
\UI  $\wbox\dneg\bbox\wbox\rdneg\wdia\wbox\dneg \bbox A \fCenter \WBOX\rdneg\wdia A $
\UI  $\wbox\dneg\bbox\wbox\rdneg\wdia\wbox\dneg \bbox A \fCenter \wbox\rdneg\wdia A $
\DisplayProof
\]
}

\item $\neg A  \wedge \neg B \fCenter \neg (A \vee B) \quad \rightsquigarrow  \quad     \wbox\dneg\bbox A  \xand  \wbox\dneg\bbox B \fCenter \wbox\rdneg\wdia(A \xor B)$,

{\fns
\[
\AXC{$A \fCenter A $}
\RightLabel{\scriptsize \WCIR}
\UIC{$\WCIR A \fCenter \WCIR A $}
\UIC{$\WCIR A \fCenter \bbox A $}
\RightLabel{\scriptsize cont}
\UIC{$\DNEG\bbox A  \fCenter \DNEG\WCIR A$}
\UIC{$\dneg\bbox A  \fCenter\DNEG\WCIR  A$}
\UIC{$\wbox\dneg\bbox A  \fCenter \WBOX\DNEG\WCIR A$}
\LeftLabel{\scriptsize $\mathrm{W}_\mathsf{L}$}
\UIC{$\wbox\dneg\bbox A \XAND \wbox\dneg\bbox B \fCenter \WBOX\DNEG\WCIR A$}
\UIC{$\wbox\dneg\bbox A \xand \wbox\dneg\bbox B \fCenter  \WBOX\DNEG\WCIR  A$}
\UIC{$\BDIA(\wbox\dneg\bbox A \xand \wbox\dneg\bbox B) \fCenter  \DNEG\WCIR  A$}
\UIC{$ \WCIR  A \fCenter  \DNEG\BDIA(\wbox\dneg\bbox A \xand \wbox\dneg\bbox B)$}
\UIC{$A \fCenter  \RH\DNEG\BDIA(\wbox\dneg\bbox A \xand \wbox\dneg\bbox B)$}
\AXC{$B \fCenter B $}
\RightLabel{\scriptsize \WCIR}
\UIC{$\WCIR B \fCenter \WCIR B$}
\UIC{$\WCIR B \fCenter \bbox B $}
\RightLabel{\scriptsize cont}
\UIC{$\DNEG\bbox B  \fCenter \DNEG\WCIR B$}
\UIC{$\dneg\bbox B  \fCenter\DNEG\WCIR  B$}
\UIC{$\wbox\dneg\bbox B \fCenter \WBOX\DNEG\WCIR B$}
\LeftLabel{\scriptsize $\mathrm{W}_\mathsf{L}$}
\UIC{ $\wbox\dneg\bbox B \XAND \wbox\dneg\bbox A \fCenter \WBOX\DNEG\WCIR B$}
\LeftLabel{\scriptsize $\mathrm{E}_\mathsf{L}$}
\UIC{$\wbox\dneg\bbox A \XAND \wbox\dneg\bbox B \fCenter \WBOX\DNEG\WCIR B$}
\UIC{$\wbox\dneg\bbox A \xand \wbox\dneg\bbox B \fCenter  \WBOX\DNEG\WCIR  B$}
\UIC{$\BDIA(\wbox\dneg\bbox A \xand \wbox\dneg\bbox B) \fCenter  \DNEG\WCIR  B$}
\UIC{$ \WCIR  B \fCenter  \DNEG\BDIA(\wbox\dneg\bbox A \xand \wbox\dneg\bbox B)$}
\UIC{$B \fCenter  \RH\DNEG\BDIA(\wbox\dneg\bbox A \xand \wbox\dneg\bbox B)$}
\BIC{$A \xor B \fCenter  \RH\DNEG\BDIA(\wbox\dneg\bbox A \xand \wbox\dneg\bbox B) \XOR \BCIR\DNEG\BDIA(\wbox\dneg\bbox A \xand \wbox\dneg\bbox B)$}
\RightLabel{\scriptsize $\mathrm{C}_\mathsf{L}$}
\UIC{$A \xor B \fCenter  \BCIR\DNEG\BDIA(\wbox\dneg\bbox A \xand \wbox\dneg\bbox B)$}
\UIC{$\WCIR(A \xor B) \fCenter  \DNEG\BDIA(\wbox\dneg\bbox A \xand \wbox\dneg\bbox B)$}
\UIC{$\wdia(A \xor B) \fCenter  \DNEG\BDIA(\wbox\dneg\bbox A \xand \wbox\dneg\bbox B)$}
\UIC{$ \BDIA(\wbox\dneg\bbox A \xand \wbox\dneg\bbox B) \fCenter  \DNEG\wdia(A \xor B)$}
\UIC{$ \BDIA(\wbox\dneg\bbox A \xand \wbox\dneg\bbox B) \fCenter  \rdneg\wdia(A \xor B)$}
\UIC{$\wbox\dneg\bbox A \xand \wbox\dneg\bbox B \fCenter   \WBOX\rdneg\wdia(A \xor B)$}
\UIC{$\wbox\dneg\bbox A \xand \wbox\dneg\bbox B \fCenter   \wbox\rdneg\wdia(A \xor B)$}
\DisplayProof
\]
}
\newpage
\item {\small $ \neg\neg A  \wedge \neg\neg B \fCenter \neg\neg (A \wedge B)\quad\rightsquigarrow \quad\wbox\dneg\bbox\wbox\rdneg\wdia A  \xand   \wbox\dneg\bbox\wbox\rdneg\wdia B \fCenter \wbox\rdneg\wdia\wbox\dneg\bbox (A \xand B) $},
\vspace{0.5cm}
{\fns
\[
\AXC{$A \fCenter A $}
\RightLabel{\scriptsize \WCIR}
\UIC{$\WCIR A \fCenter \WCIR A $}
\UIC{$\wdia A \fCenter \WCIR A $}
\RightLabel{\scriptsize cont}
\UIC{$\DNEG\WCIR A  \fCenter \DNEG\wdia A$}
\UIC{$\DNEG\WCIR A  \fCenter \rdneg\wdia A$}
\RightLabel{\scriptsize {\WCIR}\WBOX}
\UIC{$\DNEG\WCIR A  \fCenter \WCIR\WBOX\rdneg\wdia A$}
\UIC{$\LH\DNEG\WCIR A  \fCenter \WBOX\rdneg\wdia A$}
\UIC{$\LH\DNEG\WCIR A  \fCenter \wbox\rdneg\wdia A$}
\UIC{$\DNEG\WCIR A  \fCenter \WCIR\wbox\rdneg\wdia A$}
\UIC{$\DNEG\WCIR A  \fCenter \bbox\wbox\rdneg\wdia A$}
\UIC{$\DNEG\bbox\wbox\rdneg\wdia A  \fCenter\WCIR A$}
\UIC{$\dneg\bbox\wbox\rdneg\wdia A  \fCenter \WCIR A$}
\UIC{$\wbox\dneg\bbox\wbox\rdneg\wdia A  \fCenter \WBOX\WCIR A$}
\LeftLabel{\scriptsize $\mathrm{W}_\mathsf{L}$}
\UIC{$\wbox\dneg\bbox\wbox\rdneg\wdia A  \XAND \wbox\dneg\bbox\wbox\rdneg\wdia  B \fCenter \WBOX\WCIR A$}
\UIC{$\wbox\dneg\bbox\wbox\rdneg\wdia A  \xand \wbox\dneg\bbox\wbox\rdneg\wdia  B \fCenter \WBOX\WCIR A$}
\UIC{$\BDIA(\wbox\dneg\bbox\wbox\rdneg\wdia A  \xand \wbox\dneg\bbox\wbox\rdneg\wdia  B) \fCenter \WCIR A$}
\UIC{$\LH\BDIA(\wbox\dneg\bbox\wbox\rdneg\wdia A  \xand \wbox\dneg\bbox\wbox\rdneg\wdia  B) \fCenter A$}
\AXC{$B \fCenter B$}
\RightLabel{\scriptsize \WCIR}
\UIC{$\WCIR B \fCenter \WCIR B $}
\UIC{$\wdia B \fCenter \WCIR B $}
\RightLabel{\scriptsize cont}
\UIC{$\DNEG\WCIR B  \fCenter \DNEG\wdia B$}
\UIC{$\DNEG\WCIR B  \fCenter \rdneg\wdia B$}
\RightLabel{\scriptsize {\WCIR}\WBOX}
\UIC{$\DNEG\WCIR B  \fCenter \WCIR\WBOX\rdneg\wdia B$}
\UIC{$\LH\DNEG\WCIR B  \fCenter \WBOX\rdneg\wdia B$}
\UIC{$\LH\DNEG\WCIR B  \fCenter \wbox\rdneg\wdia B$}
\UIC{$\DNEG\WCIR B  \fCenter \WCIR\wbox\rdneg\wdia B$}
\UIC{$\DNEG\WCIR B  \fCenter \bbox\wbox\rdneg\wdia B$}
\UIC{$\DNEG\bbox\wbox\rdneg\wdia B  \fCenter\WCIR B$}
\UIC{$\dneg\bbox\wbox\rdneg\wdia B  \fCenter \WCIR B$}
\UIC{$\wbox\dneg\bbox\wbox\rdneg\wdia B  \fCenter \WBOX\WCIR B$}
\LeftLabel{\scriptsize $\mathrm{W}_\mathsf{L}$}
\UIC{$\wbox\dneg\bbox\wbox\rdneg\wdia B  \XAND \wbox\dneg\bbox\wbox\rdneg\wdia A \fCenter \WBOX\WCIR B$}
\LeftLabel{\scriptsize $\mathrm{E}_\mathsf{L}$}
\UIC{$\wbox\dneg\bbox\wbox\rdneg\wdia A \XAND \wbox\dneg\bbox\wbox\rdneg\wdia B  \fCenter \WBOX\WCIR B$}
\UIC{$\wbox\dneg\bbox\wbox\rdneg\wdia B  \xand \wbox\dneg\bbox\wbox\rdneg\wdia A \fCenter \WBOX\WCIR B$}
\UIC{$\BDIA(\wbox\dneg\bbox\wbox\rdneg\wdia A \xand \wbox\dneg\bbox\wbox\rdneg\wdia B) \fCenter \WCIR B$}
\UIC{$\LH\BDIA(\wbox\dneg\bbox\wbox\rdneg\wdia A \xand \wbox\dneg\bbox\wbox\rdneg\wdia B) \fCenter B$}
\BIC{$\LH\BDIA(\wbox\dneg\bbox\wbox\rdneg\wdia A \xand \wbox\dneg\bbox\wbox\rdneg\wdia B) \XAND \LH\BDIA(\wbox\dneg\bbox\wbox\rdneg\wdia A \xand \wbox\dneg\bbox\wbox\rdneg\wdia B) \fCenter A \xand B$}
\LeftLabel{\scriptsize $\mathrm{C}_\mathsf{L}$}
\UIC{$\LH\BDIA(\wbox\dneg\bbox\wbox\rdneg\wdia A \xand \wbox\dneg\bbox\wbox\rdneg\wdia B) \fCenter A \xand B$}
\UIC{$\BDIA(\wbox\dneg\bbox\wbox\rdneg\wdia A \xand \wbox\dneg\bbox\wbox\rdneg\wdia B) \fCenter \WCIR(A \xand B)$}
\UIC{$\BDIA(\wbox\dneg\bbox\wbox\rdneg\wdia A \xand \wbox\dneg\bbox\wbox\rdneg\wdia B) \fCenter \bbox(A \xand B)$}
\RightLabel{\scriptsize cont}
\UIC{$\DNEG\bbox(A \xand B) \fCenter \DNEG\BDIA(\wbox\dneg\bbox\wbox\rdneg\wdia A \xand \wbox\dneg\bbox\wbox\rdneg\wdia B)$}
\UIC{$\dneg\bbox(A \xand B) \fCenter \DNEG\BDIA(\wbox\dneg\bbox\wbox\rdneg\wdia A \xand \wbox\dneg\bbox\wbox\rdneg\wdia B)$}
\UIC{$\wbox\dneg\bbox(A \xand B) \fCenter \WBOX\DNEG\BDIA(\wbox\dneg\bbox\wbox\rdneg\wdia A \xand \wbox\dneg\bbox\wbox\rdneg\wdia B)$}
\RightLabel{\scriptsize \WCIR}
\UIC{$\WCIR\wbox\dneg\bbox(A \xand B) \fCenter \WCIR\WBOX\DNEG\BDIA(\wbox\dneg\bbox\wbox\rdneg\wdia A \xand \wbox\dneg\bbox\wbox\rdneg\wdia B)$}
\RightLabel{\scriptsize {\WCIR}\WBOX}
\UIC{$\WCIR\wbox\dneg\bbox(A \xand B) \fCenter \DNEG\BDIA(\wbox\dneg\bbox\wbox\rdneg\wdia A \xand \wbox\dneg\bbox\wbox\rdneg\wdia B)$}
\UIC{$\wdia\wbox\dneg\bbox(A \xand B) \fCenter \DNEG\BDIA(\wbox\dneg\bbox\wbox\rdneg\wdia A \xand \wbox\dneg\bbox\wbox\rdneg\wdia B)$}
\UIC{$\BDIA(\wbox\dneg\bbox\wbox\rdneg\wdia A \xand \wbox\dneg\bbox\wbox\rdneg\wdia B) \fCenter \DNEG\wdia\wbox\dneg\bbox(A \xand B)$}
\UIC{$\BDIA(\wbox\dneg\bbox\wbox\rdneg\wdia A \xand \wbox\dneg\bbox\wbox\rdneg\wdia B) \fCenter \rdneg\wdia\wbox\dneg\bbox(A \xand B)$}
\UIC{$\wbox\dneg\bbox\wbox\rdneg\wdia A \xand \wbox\dneg\bbox\wbox\rdneg\wdia B \fCenter \WBOX\rdneg\wdia\wbox\dneg\bbox(A \xand B)$}
\UIC{$\wbox\dneg\bbox\wbox\rdneg\wdia A \xand \wbox\dneg\bbox\wbox\rdneg\wdia B \fCenter \wbox\rdneg\wdia\wbox\dneg\bbox(A \xand B)$}
\DisplayProof
\]
}
\newpage
 As to the characterizing axioms of LQM and UQM:
\item $ A \fCenter \neg\neg A \quad\rightsquigarrow \quad A \fCenter \wbox\rdneg\wdia\wbox\dneg\bbox A $ and $\neg\neg A \fCenter A \quad\rightsquigarrow \quad \wbox\dneg\bbox\wbox\rdneg\bbox A  \fCenter A$,

\[
\AX $A \fCenter  A$
\RL{\scriptsize LQM}
\UI $A \fCenter \WBOX\WCIR A$
\UI $\BDIA A \fCenter \WCIR A$
\UI $\BDIA A \fCenter \bbox A$
\RL{\scriptsize cont}
\UI $\DNEG\bbox A \fCenter \DNEG\BDIA A$
\UI $\dneg\bbox A \fCenter \DNEG\BDIA A$
\UI $\wbox\dneg\bbox A \fCenter \WBOX\DNEG\BDIA A$
\RL{\scriptsize $\WCIR$}
\UI $\WCIR\wbox\dneg\bbox A \fCenter \WCIR\WBOX\DNEG\BDIA A$
\RL{\scriptsize $\WCIR\WBOX$}
\UI $\WCIR\wbox\dneg\bbox A \fCenter \DNEG\BDIA A$
\UI $\wdia\wbox\dneg\bbox A \fCenter \DNEG\BDIA A$
\UI $\BDIA A \fCenter \DNEG\wdia\wbox\dneg\bbox A$
\UI $\BDIA A \fCenter \rdneg\wdia\wbox\dneg\bbox A$
\UI $A \fCenter \WBOX\rdneg\wdia\wbox\dneg\bbox A$
\UI $A \fCenter \wbox\rdneg\wdia\wbox\dneg\bbox A$
\DP
\quad\quad
\AX $A \fCenter A$
\RL{\scriptsize $\WCIR$}
\UI $\WCIR A  \fCenter \WCIR A$
\UI $\bbox A  \fCenter \WCIR A$
\RightLabel{\scriptsize cont}
\UI $\DNEG\WCIR A  \fCenter \DNEG\bbox A$
\UI $\DNEG\WCIR A  \fCenter \rdneg\bbox A$
\RightLabel{\scriptsize {\WCIR}\WBOX}
\UI $\DNEG\WCIR A  \fCenter \WCIR\WBOX\rdneg\bbox A$
\UI $\LH\DNEG\WCIR A  \fCenter \WBOX\rdneg\bbox A$
\UI $\LH\DNEG\WCIR A  \fCenter \wbox\rdneg\bbox A$
\UI $\DNEG\WCIR A  \fCenter \WCIR\wbox\rdneg\bbox A$
\UI $\DNEG\WCIR A  \fCenter \bbox\wbox\rdneg\bbox A$
\UI $\DNEG\bbox\wbox\rdneg\bbox A  \fCenter \WCIR A$
\UI $\dneg{\bbox\wbox}\rdneg\bbox A  \fCenter \WCIR A$
\UI $\wbox\dneg{\bbox\wbox}\rdneg\bbox A  \fCenter \WBOX\WCIR A$
\UI $\BDIA\wbox\dneg{\bbox\wbox}\rdneg\bbox A  \fCenter \WCIR A$
\UI $\LH\BDIA\wbox\dneg{\bbox\wbox}\rdneg\bbox A  \fCenter A$
\RL{\scriptsize UQM}
\UI $\wbox\dneg{\bbox\wbox}\rdneg\bbox A  \fCenter A$
\DP
\]

As to the characterizing axiom of AP: 
\item $ \neg A  \wedge A \fCenter \bot \quad\rightsquigarrow \quad \wbox\dneg\bbox A \wedge  A \fCenter \bot$,

\[
\AX $A \fCenter A$
\RL{\scriptsize $\WCIR$}
\UI $\WCIR A\fCenter \WCIR A$
\UI $\WCIR A\fCenter \bbox A$
\UI $ \dneg\bbox A \fCenter \DNEG\WCIR A$
\UI $ \wbox\dneg\bbox A \fCenter \WBOX\DNEG\WCIR A$
\RL{\scriptsize AP}
\UI $A \XAND \wbox\dneg\bbox A \fCenter \XBOT$
\UI $ \wbox\dneg\bbox A \fCenter A \XRARR \XBOT$
\UI $ A \XAND \wbox\dneg\bbox A \fCenter \XBOT$
\LL{\scriptsize $\mathrm{E}_\mathsf{L}$}
\UI $ \wbox\dneg\bbox A \XAND  A \fCenter \XBOT$
\UI $ \wbox\dneg\bbox A \XAND  A \fCenter \bot$
\UI $ \wbox\dneg\bbox A \wedge  A \fCenter \bot$
\DP
\]

\newpage
As to the  characterizing axiom of DP: 
\item $\neg A  \wedge \neg\neg A \fCenter \bot \quad\rightsquigarrow \quad \wbox\dneg\bbox A \wedge \wbox\dneg\bbox\wbox\rdneg\wdia A \fCenter \bot$,

{\fns
\[
\AXC{$A \fCenter A$}
\RL{\scriptsize $\WCIR$}
\UIC{$\WCIR A \fCenter \WCIR A$}
\UIC{$\wdia A \fCenter \WCIR A$}
\UIC{$\wdia A \fCenter \bbox A$}
\RL{\scriptsize $\mathrm{W}_\mathsf{L}$}
\UIC{$\wdia A  \DAND \DBOT \fCenter  \bbox A$}
\UIC{$ \DBOT \fCenter  \DNEG\wdia A    \DOR \bbox A $}
\RL{\scriptsize $\mathrm{E}_\mathsf{L}$}
\UIC{$ \DBOT \fCenter \bbox A \DOR  \DNEG\wdia A   $}
\UIC{$\DNEG\bbox A \DAND \DBOT \fCenter \DNEG\wdia A $}
\RL{\scriptsize $\DBOT$}
\UIC{$\DNEG\bbox A \fCenter \DNEG\wdia A $}
\UIC{$\DNEG\bbox A \fCenter \rdneg\wdia A $}
\UIC{$\dneg\bbox A \fCenter \rdneg\wdia A $}
\RL{\scriptsize $\WCIR\WBOX$}
\UIC{$\dneg\bbox A \fCenter \WCIR\WBOX\rdneg\wdia A $}
\UIC{$\LH\dneg\bbox A \fCenter \WBOX\rdneg\wdia A $}
\UIC{$\LH\dneg\bbox A \fCenter \wbox\rdneg\wdia A $}
\UIC{$\dneg\bbox A \fCenter \WCIR\wbox\rdneg\wdia A$}
\UIC{$\dneg\bbox A \fCenter \bbox\wbox\rdneg\wdia A$}
\UIC{$\wbox\dneg\bbox A \fCenter \WBOX\bbox\wbox\rdneg\wdia A$}
\UIC{$\BDIA\wbox\dneg\bbox A \fCenter \bbox\wbox\rdneg\wdia A$}
\RL{\scriptsize $\mathrm{W}_\mathsf{L}$}
\UIC{$\BDIA\wbox\dneg\bbox A \fCenter \bbox\wbox\rdneg\wdia A \DOR \DBOT$}
\UIC{$\DNEG\bbox\wbox\rdneg\wdia A \DAND \BDIA\wbox\dneg\bbox A  \fCenter \DBOT $}
\LL{\scriptsize $\mathrm{E}_\mathsf{L}$}
\UIC{$\BDIA\wbox\dneg\bbox A \DAND \DNEG\bbox\wbox\rdneg\wdia A \fCenter \DBOT $}
\UIC{$\DNEG\bbox\wbox\rdneg\wdia A \fCenter \DNEG\BDIA\wbox\dneg\bbox A \DOR \DBOT $}
\UIC{$\dneg\bbox\wbox\rdneg\wdia A \fCenter \DNEG\BDIA\wbox\dneg\bbox A \DOR \DBOT $}
\UIC{$\wbox\dneg\bbox\wbox\rdneg\wdia A \fCenter \WBOX(\DNEG\BDIA\wbox\dneg\bbox A \DOR \DBOT)  $}
\UIC{$\BDIA\wbox\dneg\bbox\wbox\rdneg\wdia A \fCenter \DNEG\BDIA\wbox\dneg\bbox A \DOR \DBOT  $}
\UIC{$\BDIA\wbox\dneg\bbox A \DAND \BDIA\wbox\dneg\bbox\wbox\rdneg\wdia A \fCenter \DBOT$}
\UIC{$\BDIA\wbox\dneg\bbox\wbox\rdneg\wdia A \fCenter \BDIA\wbox\dneg\bbox A \XRARR \DBOT$}
\UIC{$\wbox\dneg\bbox\wbox\rdneg\wdia A \fCenter \WBOX(\BDIA\wbox\dneg\bbox A \XRARR \DBOT)$}
\LL{\scriptsize $\mathrm{W}_\mathsf{L}$}
\UIC{$\wbox\dneg\bbox\wbox\rdneg\wdia A \XAND \wbox\dneg\bbox A    \fCenter \WBOX(\BDIA\wbox\dneg\bbox A \XRARR \DBOT)$}
\LL{\scriptsize $\mathrm{E}_\mathsf{L}$}
\UIC{$\wbox\dneg\bbox A \XAND \wbox\dneg\bbox\wbox\rdneg\wdia A      \fCenter \WBOX(\BDIA\wbox\dneg\bbox A \XRARR \DBOT)$}
\UIC{$\BDIA(\wbox\dneg\bbox A \XAND \wbox\dneg\bbox\wbox\rdneg\wdia A)   \fCenter \BDIA\wbox\dneg\bbox A \XRARR \DBOT$}
\UIC{$ \BDIA\wbox\dneg\bbox A \DAND  \BDIA(\wbox\dneg\bbox A \XAND \wbox\dneg\bbox\wbox\rdneg\wdia A)    \fCenter \DBOT$}
\UIC{$ \BDIA\wbox\dneg\bbox A \fCenter  \BDIA(\wbox\dneg\bbox A \XAND \wbox\dneg\bbox\wbox\rdneg\wdia A)  \XRARR \DBOT$}
\UIC{$ \wbox\dneg\bbox A \fCenter  \WBOX(\BDIA(\wbox\dneg\bbox A \XAND \wbox\dneg\bbox\wbox\rdneg\wdia A)  \XRARR \DBOT)$}
\LL{\scriptsize $\mathrm{W}_\mathsf{L}$}
\UIC{$ \wbox\dneg\bbox A  \XAND \wbox\dneg\bbox\wbox\rdneg\wdia A \fCenter  \WBOX(\BDIA(\wbox\dneg\bbox A \XAND \wbox\dneg\bbox\wbox\rdneg\wdia A)  \XRARR \DBOT)$}
\UIC{$\BDIA(\wbox\dneg\bbox A  \XAND \wbox\dneg\bbox\wbox\rdneg\wdia A) \fCenter  \BDIA(\wbox\dneg\bbox A \XAND \wbox\dneg\bbox\wbox\rdneg\wdia A)  \XRARR \DBOT)$}
\UIC{$\BDIA(\wbox\dneg\bbox A \XAND \wbox\dneg\bbox\wbox\rdneg\wdia A)  \DAND \BDIA(\wbox\dneg\bbox A  \XAND \wbox\dneg\bbox\wbox\rdneg\wdia A) \fCenter \DBOT$}
\LL{\scriptsize $\mathrm{C}_\mathsf{D}$}
\UIC{$\BDIA(\wbox\dneg\bbox A \XAND \wbox\dneg\bbox\wbox\rdneg\wdia A) \fCenter \DBOT$}
\UIC{$\wbox\dneg\bbox A \XAND \wbox\dneg\bbox\wbox\rdneg\wdia A \fCenter \WBOX\DBOT$}
\RL{\scriptsize $\WBOX\DBOT$}
\UIC{$\wbox\dneg\bbox A \XAND \wbox\dneg\bbox\wbox\rdneg\wdia A \fCenter \XBOT$}
\UIC{$\wbox\dneg\bbox A \XAND \wbox\dneg\bbox\wbox\rdneg\wdia A \fCenter \bot$}
\UIC{$\wbox\dneg\bbox A \wedge \wbox\dneg\bbox\wbox\rdneg\wdia A \fCenter \bot$}
\DP
\]
}
\newpage
As to the  characterizing axiom of WS:
\item $ \top \fCenter  \neg\neg A  \vee \neg A  \quad\rightsquigarrow \quad  \top \fCenter \wbox\rdneg\wdia\wbox\dneg\bbox A \vee \wbox\rdneg\wdia A  $,

{\fns
\[
\AX $A \fCenter A$
\RL{\scriptsize $\WCIR$}
\UI $\WCIR A\fCenter \WCIR A$
\UI $\WCIR A\fCenter \bbox A$
\UI $\wdia A \fCenter \bbox A $
\UI $\dneg\bbox A \fCenter \DNEG\wdia A $
\UI $\dneg\bbox A \fCenter \rdneg\wdia A $
\UI $\wbox\dneg\bbox A \fCenter \WBOX\rdneg\wdia A $
\UI $\BDIA\wbox\dneg\bbox A \fCenter \rdneg\wdia A $
\RL{\scriptsize WS}
\UI $\BDIA(\WBOX\DNEG\WCIR\wbox\dneg\bbox A \XCRARR \XTOP)  \fCenter \rdneg\wdia A $
\UI $ \WBOX\DNEG\WCIR\wbox\dneg\bbox A \XCRARR \XTOP \fCenter \WBOX\rdneg\wdia A $
\UI $ \WBOX\DNEG\WCIR\wbox\dneg\bbox A \XCRARR \XTOP \fCenter \wbox\rdneg\wdia A $
\UI $\XTOP \fCenter \WBOX\DNEG\WCIR\wbox\dneg\bbox A \XOR \wbox\rdneg\wdia A $
\UI $\wbox\rdneg\wdia A \XCRARR \XTOP \fCenter \WBOX\DNEG\WCIR\wbox\dneg\bbox A $
\UI $\BDIA(\wbox\rdneg\wdia A \XCRARR \XTOP) \fCenter \DNEG\WCIR\wbox\dneg\bbox A $
\UI $ \WCIR\wbox\dneg\bbox A \fCenter \DNEG\BDIA(\wbox\rdneg\wdia A \XCRARR \XTOP)$
\UI $ \circ\wbox\dneg\bbox A \fCenter \DNEG\BDIA(\wbox\rdneg\wdia A \XCRARR \XTOP)$
\UI $\BDIA(\wbox\rdneg\wdia A \XCRARR \XTOP)  \fCenter \DNEG\circ\wbox\dneg\bbox A$
\UI $\BDIA(\wbox\rdneg\wdia A \XCRARR \XTOP) \fCenter \rdneg\circ\,\wbox\dneg\bbox A $
\UI $\wbox\rdneg\wdia A \XCRARR \XTOP \fCenter \WBOX\rdneg\wdia\wbox\dneg\bbox A $
\UI $\wbox\rdneg\wdia A \XCRARR \XTOP \fCenter \wbox\rdneg\wdia\wbox\dneg\bbox A $
\UI $ \XTOP \fCenter \wbox\rdneg\wdia\wbox\dneg\bbox A \XOR \wbox\rdneg\wdia A $
\UI $ \XTOP \fCenter \wbox\rdneg\wdia\wbox\dneg\bbox A \vee \wbox\rdneg\wdia A $
\UI $ \top \fCenter \wbox\rdneg\wdia\wbox\dneg\bbox A \vee \wbox\rdneg\wdia A $
\DP
\]
}

\end{itemize}
\end{proof}

\subsection{Conservativity}\label{ssec: conservativity}
To argue that the calculi introduced in Section \ref{sec:proper display} conservatively capture their respective logics (see Section \ref{ssec:Hilbert system}), we follow the standard proof strategy discussed in \cite{greco2016unified,GKPLori}. Let L be one of the logics of Definition \ref{def:logics},  let $\vdash_\msf{L}$ denote its syntactic consequence relation, and let $\models_{\msf{L}}$ (resp.~$\models_{\msf{HL}}$)  denote the semantic consequence relation arising from the class of the perfect (heterogeneous) algebras associated with L.  We need to show that, for all $\mathcal{L}$-formulas $A$ and $B$, if $A^\tau \vdash B^\tau$ is derivable in the display calculus D.L,  then  $A \vdash_{\msf{L}} B$. This claim can be proved using  the following facts: (a) the rules of D.L are sound w.r.t.~perfect heterogeneous  L-algebras  (cf.~Section \ref{ssec:soundness});  (b) L is complete w.r.t.~its associated class of algebras (cf.~Theorem \ref{completeness: H.SDM}); and (c)  L-algebras are equivalently presented as heterogeneous  L-algebras (cf.~Section \ref{Heterogeneous presentation}), so that the semantic consequence relations arising from each type of algebras preserve and reflect the translation (cf.~Proposition \ref{prop:consequence preserved and reflected}).  If  $A^\tau\vdash B^\tau$ is derivable in D.L,  then by (a),  $\models_{\msf{HL}} A^\tau \fCenter B^\tau$. By (c), this implies that $\models_{\msf{L}} A \fCenter B$. By (b), this implies that $A\vdash_{\msf{L}} B$, as required.

\subsection{Cut elimination and subformula property}
In the present subsection, we briefly sketch the proof of cut elimination and subformula property for all display calculi introduced in Section \ref{ssec:display calculi}. As discussed earlier on, proper display calculi have been designed so that the cut elimination and subformula property  can  be inferred from a meta-theorem, following the strategy introduced by Belnap for display calculi. The meta-theorem to which we will appeal was proved in \cite[Theorem 4.1]{TrendsXIII}.

\begin{theorem}\label{cut elimination}
Cut elimination and subformula property hold for all display calculi introduced in Section \ref{ssec:display calculi}.
\end{theorem}

\begin{proof}
All conditions in  \cite{TrendsXIII} except $\textrm{C}'_8$ are readily satisfied by inspecting the rules. Condition $\textrm{C}'_8$ requires to check that reduction steps are available for every application of the cut rule in which both cut-formulas are principal, which either remove the original cut altogether or replace it by one or more cuts on formulas of strictly lower complexity.  In what follows, we show  $\textrm{C}'_8$ for the unary connectives by induction on the complexity of cut formula.
\paragraph*{Pure type atomic propositions:}

\begin{center}
\begin{tabular}{ccc}
\bottomAlignProof
\AX$p \fCenter p$
\AX$p \fCenter p$
\BI$p \fCenter p$
\DisplayProof

 & $\rightsquigarrow$ &

\bottomAlignProof
\AX$p \fCenter p$
\DisplayProof
 \\
\end{tabular}
\end{center}

\paragraph*{Pure type constants:}

\begin{center}
\begin{tabular}{ccc}
\bottomAlignProof
\AX$\XTOP \fCenter \xtop$
\AXC{\ \ \ $\vdots$ \raisebox{1mm}{$\pi_1$}}
\noLine
\UI$\XTOP \fCenter X$
\UI$\xtop \fCenter X$
\BI$\XTOP \fCenter X$
\DisplayProof

 & $\rightsquigarrow$ &

\bottomAlignProof
\AXC{\ \ \ $\vdots$ \raisebox{1mm}{$\pi_1$}}
\noLine
\UI$\XTOP \fCenter X$
\DisplayProof
 \\
\end{tabular}
\end{center}

\noindent The cases for $\xbot$, $\dtop$, $\dbot$ are standard and similar to the one above.

\paragraph*{Pure-type unary connectives:}

\begin{center}
\begin{tabular}{ccc}
\!\!\!\!\!
\bottomAlignProof
\AXC{\ \ \ $\vdots$ \raisebox{1mm}{$\pi_1$}}
\noLine
\UI$\Gamma \fCenter \DNEG \alpha$
\UI$\Gamma \fCenter \rdneg \alpha$

\AXC{\ \ \ $\vdots$ \raisebox{1mm}{$\pi_2$}}
\noLine
\UI$\DNEG\alpha \fCenter \Delta$
\UI$\rdneg \alpha \fCenter \Delta$
\BI$\Gamma \fCenter \Delta$
\DisplayProof

 & $\rightsquigarrow$ &

\!\!\!\!\!\!\!\!\!\!\!\!\!\!\!\!\!\!\!\!
\bottomAlignProof
\AXC{\ \ \ $\vdots$ \raisebox{1mm}{$\pi_2$}}
\noLine
\UI$\DNEG\alpha \fCenter \Delta$
\UI$\DNEG\Delta \fCenter \alpha$

\AXC{\ \ \ $\vdots$ \raisebox{1mm}{$\pi_1$}}
\noLine
\UI$\Gamma \fCenter \DNEG \alpha$
\UI$\alpha \fCenter \DNEG\Gamma$
\BI$\DNEG \Delta \fCenter \DNEG\Gamma$
\RightLabel{\scriptsize cont}
\UI$\Gamma \fCenter \Delta$
\DisplayProof
 \\
\end{tabular}
\end{center}

\paragraph*{Pure-type binary connectives:}
\setlength{\parindent}{-4em}
{\small
\begin{center}
\begin{tabular}{ccc}
\!\!\!\!\!
\bottomAlignProof
\AXC{\ \ \ $\vdots$ \raisebox{1mm}{$\pi_1$}}
\noLine
\UI$X \fCenter A$
\AXC{\ \ \ $\vdots$ \raisebox{1mm}{$\pi_2$}}
\noLine
\UI$Y \fCenter B$
\BI$X \XAND Y \fCenter A \xand B$
\AXC{\ \ \ $\vdots$ \raisebox{1mm}{$\pi_3$}}
\noLine
\UIC{$A \XAND B \fCenter Z$}
\UIC{$A \xand B \fCenter Z$}
\BIC{$X \XAND Y \fCenter Z$}
\DisplayProof

 & $\rightsquigarrow$ &

\!\!\!\!\!\!\!\!\!\!\!\!\!\!\!\!\!\!\!\!
\bottomAlignProof
\AXC{\ \ \ $\vdots$ \raisebox{1mm}{$\pi_1$}}
\noLine
\UIC{$X \fCenter A$}
\AXC{\ \ \ $\vdots$ \raisebox{1mm}{$\pi_2$}}
\noLine
\UI$Y \fCenter B$
\AXC{\ \ \ $\vdots$ \raisebox{1mm}{$\pi_3$}}
\noLine
\UIC{$A \XAND B \fCenter Z$}
\UIC{$B \fCenter A \XRARR Z$}
\BIC{$Y \fCenter A \XRARR Z$}
\UIC{$A \XAND Y \fCenter Z$}
\UIC{$Y \XAND A \fCenter Z$}
\UIC{$A \fCenter Y \XRARR Z$}
\BIC{$X \fCenter Y \XRARR Z$}
\UIC{$Y \XAND X \fCenter Z$}
\UIC{$X \XAND Y \fCenter Z$}
\DisplayProof
 \\
\end{tabular}
\end{center}
}
\noindent The cases for $A \xor B$, $\alpha \dand \beta$, $\alpha \dor \beta$ are standard and similar to the one above.

\paragraph*{Multi-type unary connectives:}

\begin{center}
\begin{tabular}{ccc}
\bottomAlignProof
\AXC{\ \ \ $\vdots$ \raisebox{1mm}{$\pi_1$}}
\noLine
\UI$X \fCenter \WBOX\alpha$
\UI$X \fCenter \wbox\alpha$
\AXC{\ \ \ $\vdots$ \raisebox{1mm}{$\pi_2$}}
\noLine
\UI$\alpha \fCenter \Delta$
\UI$\wbox\alpha \fCenter \WBOX \Delta$
\BI$X \fCenter \WBOX\Delta$
\DisplayProof

 & $\rightsquigarrow$ &

\!\!\!\!\!\!\!
\bottomAlignProof
\AXC{\ \ \ $\vdots$ \raisebox{1mm}{$\pi_1$}}
\noLine
\UI$X \fCenter \WBOX\alpha$
\UI$\BDIA X \fCenter \alpha$
\AXC{\ \ \ $\vdots$ \raisebox{1mm}{$\pi_2$}}
\noLine
\UI$\alpha \fCenter \Delta$
\BI$\BDIA X \fCenter \Delta$
\UI$X \fCenter \WBOX\Delta$
\DisplayProof
 \\
\end{tabular}
\end{center}

\begin{center}
\begin{tabular}{ccc}
\bottomAlignProof
\AXC{\ \ \ $\vdots$ \raisebox{1mm}{$\pi_1$}}
\noLine
\UI$\Gamma \fCenter \WCIR A$
\UI$\Gamma \fCenter \bbox A$
\AXC{\ \ \ $\vdots$ \raisebox{1mm}{$\pi_2$}}
\noLine
\UI$\WCIR A\fCenter \Delta$
\UI$\bbox A \fCenter  \Delta$
\BI$\Gamma \fCenter  \Delta$
\DisplayProof

 & $\rightsquigarrow$ &

\!\!\!\!\!\!\!
\bottomAlignProof
\AXC{\ \ \ $\vdots$ \raisebox{1mm}{$\pi_1$}}
\noLine
\UI$\Gamma \fCenter \WCIR A$
\UI$\LH\Gamma \fCenter A$
\AXC{\ \ \ $\vdots$ \raisebox{1mm}{$\pi_2$}}
\noLine
\UI$\WCIR A\fCenter \Delta$
\UI$ A\fCenter \RH \Delta$

\BI$\LH\Gamma \fCenter \RH \Delta$
\RightLabel{\scriptsize $\bullet$}
\UI$\Gamma \fCenter  \Delta$
\DisplayProof
\end{tabular}
\end{center}

\end{proof}

\bibliography{dissertation-3}

\begin{thebibliography}{10}

\bibitem{BGPTW}
Marta B\'{i}lkov\'{a}, Giuseppe Greco, Alessandra Palmigiano, Apostolos
  Tzimoulis, and Nachoem Wijnberg.
\newblock The logic of resources and capabilities.
\newblock {\em Review of Symbolic Logic}, forthcoming. ArXiv preprint
  1608.02222.

\bibitem{celani1999distributive}
Sergio~Arturo Celani.
\newblock Distributive lattices with a negation operator.
\newblock {\em Mathematical Logic Quarterly}, 45(2):207--218, 1999.

\bibitem{celani2007representation}
Sergio~Arturo Celani.
\newblock Representation for some algebras with a negation operator.
\newblock {\em Contributions to Discrete Mathematics}, 2(2):205--213, 2007.

\bibitem{conradie2016algorithmic}
Willem Conradie and Alessandra Palmigiano.
\newblock {Algorithmic correspondence and canonicity for non-distributive
  logics}.
\newblock Submitted. ArXiv preprint 1603.08515.

\bibitem{PDL}
Sabine Frittella, Giuseppe Greco, Alexander Kurz, and Alessandra Palmigiano.
\newblock Multi-type display calculus for propositional dynamic logic.
\newblock {\em Journal of Logic and Computation}, 26 (6):2067--2104, 2016.

\bibitem{TrendsXIII}
Sabine Frittella, Giuseppe Greco, Alexander Kurz, Alessandra Palmigiano, and
  Vlasta Sikimi\'{c}.
\newblock Multi-type sequent calculi.
\newblock {\em Proceedings Trends in Logic XIII, A. Indrzejczak, J. Kaczmarek,
  M. Zawidski eds}, 13:81--93, 2014.

\bibitem{Multitype}
Sabine Frittella, Giuseppe Greco, Alexander Kurz, Alessandra Palmigiano, and
  Vlasta Sikimi\'{c}.
\newblock A multi-type display calculus for dynamic epistemic logic.
\newblock {\em Journal of Logic and Computation}, 26 (6):2017--2065, 2016.

\bibitem{inquisitive}
Sabine Frittella, Giuseppe Greco, Alessandra Palmigiano, and Fan Yang.
\newblock A multi-type calculus for inquisitive logic.
\newblock In Jouko V{\"a}{\"a}n{\"a}nen, {\AA}sa Hirvonen, and Ruy de~Queiroz,
  editors, {\em Logic, Language, Information, and Computation: 23rd
  International Workshop, WoLLIC 2016, Puebla, Mexico, August 16-19th, 2016.
  Proceedings}, LNCS 9803, pages 215--233. Springer, 2016.

\bibitem{gehrke2001bounded}
Mai Gehrke and John Harding.
\newblock Bounded lattice expansions.
\newblock {\em Journal of Algebra}, 238(1):345--371, 2001.

\bibitem{gehrke2000monotone}
Mai Gehrke and Bjarni J{\'o}nsson.
\newblock Monotone bouded distributive lattice expansions.
\newblock {\em Mathematica japonicae}, 52(2):197--213, 2000.

\bibitem{gehrke2004bounded}
Mai Gehrke and Bjarni J{\'o}nsson.
\newblock Bounded distributive lattice expansions.
\newblock {\em Mathematica Scandinavica}, pages 13--45, 2004.

\bibitem{GNV05}
Mai Gehrke, Hideo Nagahashi, and Yde Venema.
\newblock A sahlqvist theorem for distributive modal logic.
\newblock {\em Annals of pure and applied logic}, 131(1-3):65--102, 2005.

\bibitem{GKPLori}
Giuseppe Greco, Alexander Kurz, and Alessandra Palmigiano.
\newblock Dynamic epistemic logic displayed.
\newblock In Huaxin Huang, Davide Grossi, and Olivier Roy, editors, {\em
  Proceedings of the 4th International Workshop on Logic, Rationality and
  Interaction (LORI-4)}, volume 8196 of {\em LNCS}, 2013.

\bibitem{greco2017multi}
Giuseppe Greco, Fei Liang, M~Andrew Moshier, and Alessandra Palmigiano.
\newblock Multi-type display calculus for semi de morgan logic.
\newblock In {\em International Workshop on Logic, Language, Information, and
  Computation}, pages 199--215. Springer, 2017.

\bibitem{greco2016unified}
Giuseppe Greco, Minghui Ma, Alessandra Palmigiano, Apostolos Tzimoulis, and
  Zhiguang Zhao.
\newblock Unified correspondence as a proof-theoretic tool.
\newblock {\em Journal of Logic and Computation}, 2016. doi:
  10.1093/logcom/exw022.

\bibitem{GrecoPalmigianoLatticeLogic}
Giuseppe Greco and Alessandra Palmigiano.
\newblock Lattice logic properly displayed.
\newblock In {\em International Workshop on Logic, Language, Information, and
  Computation}, pages 153--169. Springer, 2017.

\bibitem{linearlogPdisplayed}
Giuseppe Greco and Alessandra Palmigiano.
\newblock Linear logic properly displayed.
\newblock Submitted. ArXiv preprint:1611.04181.

\bibitem{hobby1996semi}
David Hobby.
\newblock Semi-{D}e {M}organ algebras.
\newblock {\em Studia Logica}, 56(1-2):151--183, 1996.

\bibitem{Palma2005}
C{\^a}ndida Palma.
\newblock Semi {D}e {M}organ algebras.
\newblock Dissertation, the University of Lisbon, 2005.

\bibitem{palmigiano2016sahlqvist}
Alessandra Palmigiano, Sumit Sourabh, and Zhiguang Zhao.
\newblock Sahlqvist theory for impossible worlds.
\newblock {\em Journal of Logic and Computation}, page exw014, 2016.

\bibitem{priest2002paraconsistent}
Graham Priest.
\newblock Paraconsistent logic.
\newblock In {\em Handbook of philosophical logic}, pages 287--393. Springer,
  2002.

\bibitem{sankappanavar1987semi}
Hanamantagouda~P Sankappanavar.
\newblock Semi-{D}e {M}organ algebras.
\newblock {\em The Journal of symbolic logic}, 52(03):712--724, 1987.

\bibitem{Wa98}
Heinrich Wansing.
\newblock {\em Displaying Modal Logic}.
\newblock Kluwer, 1998.

\end{thebibliography}
\bibliographystyle{plain}

\appendix
\section{Analytic inductive inequalities}
\label{sec:analytic inductive ineq}
In the present section, we  specialize the definition of {\em analytic inductive inequalities} (cf.\ \cite{greco2016unified}) to the  multi-type language $\mathcal{L}_\textrm{MT}$, in the types $\mathsf{DL}$ and $\mathsf{K}$,  defined in Section \ref{sec: multi-type language}  and reported below for the reader's convenience.
\begin{align*}
\mathsf{DL}\ni  A ::= & \,p \mid \, {\wbox}\alpha \mid \xtop \mid \xbot \mid A \xand A \mid A \xor A \\
\mathsf{K}\ni \alpha ::=&\, {\circ} A \mid \dtop \mid \dbot \mid\,\rdneg\alpha \mid \alpha \dor \alpha \mid  \alpha \dand \alpha
\end{align*}
We will make use of the following auxiliary definition: an {\em order-type} over $n\in \mathbb{N}$  is an $n$-tuple $\epsilon\in \{1, \partial\}^n$. For every order type $\epsilon$, we denote its {\em opposite} order type by $\epsilon^\partial$, that is, $\epsilon^\partial(i) = 1$ iff $\epsilon(i)=\partial$ for every $1 \leq i \leq n$.
The connectives of the language above are grouped together  into the  families $\mathcal{F}: = \mathcal{F}_{\mathsf{DL}}\cup \mathcal{F}_{\mathsf{K}}\cup \mathcal{F}_{\textrm{MT}}$, $\mathcal{G}: = \mathcal{G}_{\mathsf{DL}}\cup \mathcal{G}_{\mathsf{K}} \cup  \mathcal{G}_{\textrm{MT}}$, and $\mathcal{H}: = \mathcal{H}_{\mathsf{DL}}\cup \mathcal{H}_{\mathsf{K}}\cup \mathcal{H}_{\textrm{MT}}$ defined as follows:
\begin{center}
\begin{tabular}{lcl}
$\mathcal{F}_{\mathsf{DL}}: = \{\wedge, \top\}$&$ \mathcal{G}_{\mathsf{DL}}:= \{\xor, \xbot\}$&$\mathcal{H}_{\mathsf{DL}}: = \varnothing$\\
$\mathcal{F}_{\mathsf{K}}: = \{\dand,\dtop\}$ & $\mathcal{G}_{\mathsf{K}}: = \{\dor,\dbot\}$&$ \mathcal{H}_{\mathsf{K}}:= \{\dneg\}$\\
$\mathcal{F}_{\textrm{MT}}: = \varnothing$ & $\mathcal{G}_{\textrm{MT}}: = \{\wbox\}$&$\mathcal{H}_{\textrm{MT}}:= \{\circ\}$\\

\end{tabular}
\end{center}
For any $\ell \in \mathcal{F} \cup \mathcal{G} \cup\mathcal{H}$, we let $n_\ell \in \mathbb{N}$ denote the arity of $\ell$, and the order-type $\epsilon_\ell$ on $n_\ell$ indicates whether the $i$th coordinate of $\ell$ is positive ($\epsilon_\ell(i) = 1$) or  negative ($\epsilon_\ell(i) = \partial$). The order-theoretic motivation for this partition is that the algebraic interpretations of $\mathcal{F}$-connectives (resp.\ $\mathcal{G}$-connectives), preserve finite joins (resp.\ meets) in each positive coordinate and reverse finite meets (resp.\ joins) in each negative coordinate, while the algebraic interpretations of $\mathcal{H}$-connectives, preserve both finite joins and meets in each positive coordinate and reverse both finite meets and joins in each negative coordinate.

				
For any term $s(p_1,\ldots p_n)$, any order type $\epsilon$ over $n$, and any $1 \leq i \leq n$, an \emph{$\epsilon$-critical node} in a signed generation tree of $s$ is a leaf node $+p_i$ with $\epsilon(i) = 1$ or $-p_i$ with $\epsilon(i) = \partial$. An $\epsilon$-{\em critical branch} in the tree is a branch ending in an $\epsilon$-critical node. For any term $s(p_1,\ldots p_n)$ and any order type $\epsilon$ over $n$, we say that $+s$ (resp.\ $-s$) {\em agrees with} $\epsilon$, and write $\epsilon(+s)$ (resp.\ $\epsilon(-s)$), if every leaf in the signed generation tree of $+s$ (resp.\ $-s$) is $\epsilon$-critical.
				 We will also write $+s'\prec \ast s$ (resp.\ $-s'\prec \ast s$) to indicate that the subterm $s'$ inherits the positive (resp.\ negative) sign from the signed generation tree $\ast s$. Finally, we will write $\epsilon(s') \prec \ast s$ (resp.\ $\epsilon^\partial(s') \prec \ast s$) to indicate that the signed subtree $s'$, with the sign inherited from $\ast s$, agrees with $\epsilon$ (resp.\ with $\epsilon^\partial$).
				\begin{definition}[\textbf{Signed Generation Tree}]
					\label{def: signed gen tree}
					The \emph{positive} (resp.\ \emph{negative}) {\em generation tree} of any $\mathcal{L}_\textrm{MT}$-term $s$ is defined by labelling the root node of the generation tree of $s$ with the sign $+$ (resp.\ $-$), and then propagating the labelling on each remaining node as follows:
					For any node labelled with $\ell \in \mathcal{F}\cup \mathcal{G} \cup \mathcal{H}$ of arity $n_\ell$, and for any $1\leq i\leq n_\ell$, assign the same (resp.\ the opposite) sign to its $i$th child node if $\epsilon_\ell(i) = 1$ (resp.\ if $\epsilon_\ell(i) = \partial$). Nodes in signed generation trees are \emph{positive} (resp.\ \emph{negative}) if are signed $+$ (resp.\ $-$).
				\end{definition}
				
		\begin{definition}[\textbf{Good branch}]
					\label{def:good:branch}
					Nodes in signed generation trees will be called \emph{$\Delta$-adjoints}, \emph{syntactically left residual (SLR)}, \emph{syntactically right residual (SRR)}, and \emph{syntactically right adjoint (SRA)}, according to the specification given in Table \ref{Join:and:Meet:Friendly:Table}.
					A branch in a signed generation tree $\ast s$, with $\ast \in \{+, - \}$, is called a \emph{good branch} if it is the concatenation of two paths $P_1$ and $P_2$, one of which may possibly be of length $0$, such that $P_1$ is a path from the leaf consisting (apart from variable nodes) only of PIA-nodes\footnote{For an expanded discussion on this definition, see \cite[Remark 3.24]{palmigiano2016sahlqvist} and \cite[Remark 3.3]{conradie2016algorithmic}.}, and $P_2$ consists (apart from variable nodes) only of Skeleton-nodes.

									\begin{table}
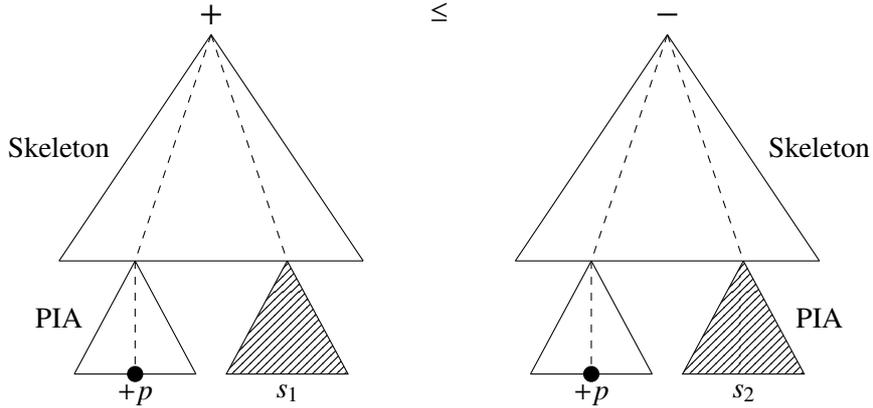
						\begin{center}
							\begin{tabular}{| c | c |}
								\hline
								Skeleton &PIA\\
								\hline
								$\Delta$-adjoints & SRA \\
								\begin{tabular}{c c c ccc}
									 $+\ $ & $\phantom{\xand}$& $\xor$ & $\dor$& $\phantom{\dand}$ & \\
									$-\ $ && $\xand$  &$\dand$ && \\
								\end{tabular}
								&
								\begin{tabular}{c c c c c c  }
									$+\ $ & \ $\xand$\  &\ $\dand$\ &\ $\circ$ \ &\ $\rdneg$\ &\ $\wbox\ $\\
									$-\ $ & \  $\xor $\  &\ $\dor$\  &\ $\circ$ \ &\ $\dneg$\   \\
								\end{tabular}
								\\
																\hline

								SLR &SRR\\
								\begin{tabular}{c c c c c c}
									$+\ $ & \ $\xand$\  &\ $\dand$\ &\ $\circ$\  &\ $\dneg$\ &  \\
									$-\ $ & \ $\xor$\  &\ $\dor$\  &\ $\circ$\ &\ $\rdneg$\ &\ $\wbox$\  \\
								\end{tabular}
								&\begin{tabular}{c c c ccc}
									 $+\ $ & $\phantom{\xand}$& $\xor$ & $\dor$& $\phantom{\dand}$ & \\
									$-\ $ && $\xand$  &$\dand$ && \\
								\end{tabular}
								\\
								\hline
							\end{tabular}
						\end{center}
						\caption{Skeleton and PIA nodes.}\label{Join:and:Meet:Friendly:Table}
						\vspace{-1em}
					\end{table}
				\end{definition}
	
				\begin{center}
		\begin{tikzpicture}
		\draw (-5,-1.5) -- (-3,1.5) node[above]{\Large$+$} ;
		\draw (-5,-1.5) -- (-1,-1.5) ;
		\draw (-3,1.5) -- (-1,-1.5);
		\draw (-5,0) node{Skeleton} ;
		\draw[dashed] (-3,1.5) -- (-4,-1.5);
		\draw[dashed] (-3,1.5) -- (-2,-1.5);
		\draw (-4,-1.5) --(-4.8,-3);
		\draw (-4.8,-3) --(-3.2,-3);
		\draw (-3.2,-3) --(-4,-1.5);
		\draw[dashed] (-4,-1.5) -- (-4,-3);
		\draw[fill] (-4,-3) circle[radius=.1] node[below]{$+p$};
		\draw
		(-2,-1.5) -- (-2.8,-3) -- (-1.2,-3) -- (-2,-1.5);
		\fill[pattern=north east lines]
		(-2,-1.5) -- (-2.8,-3) -- (-1.2,-3);
		\draw (-2,-3.25)node{$s_1$};
		\draw (-5,-2.25) node{PIA} ;
		\draw (0,1.8) node{$\leq$};
		\draw (5,-1.5) -- (3,1.5) node[above]{\Large$-$} ;
		\draw (5,-1.5) -- (1,-1.5) ;
		\draw (3,1.5) -- (1,-1.5);
		\draw (5,0) node{Skeleton} ;
		\draw[dashed] (3,1.5) -- (4,-1.5);
		\draw[dashed] (3,1.5) -- (2,-1.5);
		\draw (2,-1.5) --(2.8,-3);
		\draw (2.8,-3) --(1.2,-3);
		\draw (1.2,-3) --(2,-1.5);
		\draw[dashed] (2,-1.5) -- (2,-3);
		\draw[fill] (2,-3) circle[radius=.1] node[below]{$+p$};
		\draw
		(4,-1.5) -- (4.8,-3) -- (3.2,-3) -- (4, -1.5);
		\fill[pattern=north east lines]
		(4,-1.5) -- (4.8,-3) -- (3.2,-3) -- (4, -1.5);
		\draw (4,-3.25)node{$s_2$};
		\draw (5,-2.25) node{PIA} ;
		\end{tikzpicture}
	\end{center}

				\begin{definition}[\textbf{Analytic inductive inequalities}]
	\label{def:analytic inductive ineq}
					For any order type $\epsilon$ and any irreflexive and transitive relation $<_\Omega$ on $p_1,\ldots p_n$, the signed generation tree $*s$ $(* \in \{-, + \})$ of an $\mathcal{L}_{MT}$ term $s(p_1,\ldots p_n)$ is \emph{analytic $(\Omega, \epsilon)$-inductive} if
					\begin{enumerate}
						\item  every branch of $*s$ is good (cf.\ Definition \ref{def:good:branch});
						\item for all $1 \leq i \leq n$, every SRR-node occurring in  any $\epsilon$-critical branch with leaf $p_i$ is of the form $ \circledast(s, \beta)$ or $ \circledast(\beta, s)$, where the critical branch goes through $\beta$ and 
						\begin{enumerate}
							\item $\epsilon^\partial(s) \prec \ast s$ (cf.\ discussion before Definition \ref{def:good:branch}), and
							%
							\item $p_k <_{\Omega} p_i$ for every $p_k$ occurring in $s$ and for every $1\leq k\leq n$.
						\end{enumerate}

					\end{enumerate}
					
					We will refer to $<_{\Omega}$ as the \emph{dependency order} on the variables. An inequality $s \leq t$ is \emph{analytic $(\Omega, \epsilon)$-inductive} if the signed generation trees $+s$ and $-t$ are analytic $(\Omega, \epsilon)$-inductive. An inequality $s \leq t$ is \emph{analytic inductive} if is analytic $(\Omega, \epsilon)$-inductive for some $\Omega$ and $\epsilon$.
				\end{definition}
								
In each setting in which they are defined, analytic inductive inequalities are a subclass of inductive inequalities (cf.~\cite[Definition 16]{greco2016unified}). In their turn, inductive inequalities are {\em canonical} (that is, preserved under canonical extensions, as defined in each setting).

\end{document}